\documentclass[a4paper]{amsart}
\usepackage[latin9]{inputenc}
\usepackage{xcolor}
\usepackage{pdfcolmk}
\usepackage{verbatim}
\usepackage{amstext}
\usepackage{amsthm}
\usepackage{amssymb}
\usepackage{stmaryrd}
\PassOptionsToPackage{normalem}{ulem}
\usepackage{ulem}
\usepackage[unicode=true,pdfusetitle,
 bookmarks=true,bookmarksnumbered=false,bookmarksopen=false,
 breaklinks=false,pdfborder={0 0 1},backref=false,colorlinks=false]
 {hyperref}

\makeatletter

\providecolor{lyxadded}{rgb}{0,0,1}
\providecolor{lyxdeleted}{rgb}{1,0,0}

\DeclareRobustCommand{\lyxsout}[1]{\ifx\\#1\else\sout{#1}\fi}

\numberwithin{equation}{section}
\numberwithin{figure}{section}
\theoremstyle{plain}
\newtheorem*{thm*}{\protect\theoremname}
\theoremstyle{plain}
\newtheorem{thm}{\protect\theoremname}
\theoremstyle{remark}
\newtheorem{rem}[thm]{\protect\remarkname}
\theoremstyle{definition}
\newtheorem{defn}[thm]{\protect\definitionname}
\theoremstyle{plain}
\newtheorem{prop}[thm]{\protect\propositionname}
\theoremstyle{plain}
\newtheorem{lem}[thm]{\protect\lemmaname}

\usepackage{amssymb,amsthm,amsmath,amsfonts,amscd}
\usepackage{graphicx}
\usepackage{url}
\usepackage{color}
\usepackage[active]{srcltx}
\usepackage[matrix,arrow]{xy}
\usepackage{mathrsfs}
\usepackage{enumerate}
\usepackage{amsopn} 
\usepackage{bbm} 
\usepackage{cite}
\usepackage{hyperref}
\allowdisplaybreaks[1]
\usepackage[english]{babel}
\usepackage{bbm}
\usepackage{mdef}
\date{\today}

\makeatother

\providecommand{\definitionname}{Definition}
\providecommand{\lemmaname}{Lemma}
\providecommand{\propositionname}{Proposition}
\providecommand{\remarkname}{Remark}
\providecommand{\theoremname}{Theorem}

\begin{document}
\title[Well-posedness by noise for SCL]{Well-posedness by noise for scalar conservation laws}
\begin{abstract}
We consider stochastic scalar conservation laws with spatially inhomogeneous
flux. The regularity of the flux function with respect to its spatial
variable is assumed to be low, so that entropy solutions are not necessarily
unique in the corresponding deterministic scalar conservation law.
We prove that perturbing the system by noise leads to well-posedness.
\end{abstract}

\author{Benjamin Gess}
\address{Max-Planck Institute for Mathematics in the Sciences\\
Inselstrasse 22, 04103 Leipzig\\
Germany}
\email{bgess@mis.mpg.de}
\author{Mario Maurelli}
\address{Department of Mathematics, University of York\\
York, YO10 5DD\\
United Kingdom\\
\&\\
School of Mathematics, University of Edinburgh\\
James Clerk Maxwell Building, Peter Guthrie Tait Road, Edinburgh,
EH9 3FD\\
United Kingdom}
\email{mario.maurelli@york.ac.uk}
\thanks{M.M.\ acknowledges funding from the European Research Council under
the European Union's Seventh Framework Program (FP7/2007-2013) / ERC
grant agreement nr.\ 258237. }
\maketitle

\section{Introduction}

In this paper we prove a well-posedness by noise result for the inhomogeneous
scalar conservation laws
\begin{align}
 & du(t,x)+b(x,u(t,x))\cdot\nabla u(t,x)\,dt+\nabla u(t,x)\circ dW_{t}=0\quad\text{on }[0,T]\times\R^{d},\label{eq:intro-stoch-burgers-inhomo}\\
 & u(0,\cdot)=u_{0}.\nonumber 
\end{align}
If the vector field $b$ lacks sufficient regularity with respect
to the space variable $x$, ill-posedness can appear. For a counterexample
see \eqref{eq:model_ex} below. In contrast, the main result of this
paper (Theorem \ref{thm:well_posedness} below), here stated in a
slightly simplified version, shows that noise restores well-posedness
for some of these vector fields:
\begin{thm*}
Assume that $b\in L_{u,loc}^{\infty}(L_{x}^{\infty})\cap L_{u,loc}^{1}(W_{x,loc}^{1,1})$
and that $\div\,b\in L_{u,loc}^{1}(L_{x}^{1})\cap L_{x}^{p}(L_{u,loc}^{\infty})$
for some $p>d$, $p\le\infty$. Then, for every initial datum $u_{0}$
in $(L^{1}\cap L^{\infty})(\R^{d})$, there exists a unique entropy
solution $u$ to \eqref{eq:intro-stoch-burgers-inhomo}.
\end{thm*}
The question of regularization and well-posedness by noise for SPDE
has attracted considerable interest in recent years. One of the driving
hopes in this field is to obtain the well-posedness by noise for nonlinear
PDE arising in fluid dynamics, for which the deterministic counterpart
does not or is not known to allow unique solutions. Despite considerable
effort, only partial results in this direction could be obtained so
far, cf.~e.g.~Flandoli \cite{F11}, Flandoli, Romito \cite{FR02,FR08},
Delarue, Flandoli, Vincenzi \cite{DFV14} and the references therein.
In the linear setting (i.e.\ for $b$ independent of $u$), one of
the prominent works in this direction by Flandoli, Gubinelli, Priola
is \cite{FGP10} in which the well-posedness by noise for linear transport
equations with irregular drift has been shown. More precisely, while
weak solutions to
\begin{equation}
\partial_{t}u(t,x)+b(x)\cdot\nabla u(t,x)=0\quad\text{on }[0,T]\times\R^{d}\label{eq:intro-det-transport}
\end{equation}
(omitting the initial condition for simplicity of notation) are not
necessarily unique if $\div\,b\not\in L^{\infty}(\R^{d})$ (cf.~DiPerna,
Lions \cite{DPL89}, Ambrosio \cite{A04-1}), it has been shown in
\cite{FGP10} that weak solutions to
\begin{equation}
du(t,x)+b(x)\cdot\nabla u(t,x)\,dt+\nabla u(t,x)\circ dW_{t}=0\quad\text{on }[0,T]\times\R^{d}\label{eq:intro-stoch-transport}
\end{equation}
are unique, provided $b\in C_{b}^{\a}(\R^{d})$ for some $\a\in(0,1)$,
$\div\,b\in L^{p}(\R^{d})$ for some $p>2$ and $W_{t}$ denotes a
standard $d$-dimensional Wiener process. As pointed out in \cite{FGP10}
their result yields \textit{the first concrete example of a partial
differential equation related to fluid dynamics that may lack uniqueness
without noise, but is well-posed with a suitable noise }(cf.~\cite[p.3, l.1 ff.]{FGP10})\textit{.}
On the other hand, as observed in \cite{FGP10,F11}, in the nonlinear
setting ($d=1$ for simplicity)
\begin{equation}
\partial_{t}u(t,x)+\partial_{x}u^{2}(t,x)=0\quad\text{on }[0,T]\times\R\label{eq:intro-det-burgers-homogeneous}
\end{equation}
the same type of noise seems to be of little use, since the stochastically
perturbed equation
\begin{equation}
du(t,x)+\partial_{x}u^{2}(t,x)\,dt+\partial_{x}u(t,x)\circ dW_{t}=0\quad\text{on }[0,T]\times\R\label{eq:intro-stoch-burgers-homogeneous}
\end{equation}
reduces to the deterministic case \eqref{eq:intro-det-burgers-homogeneous}
via the transformation $v(t,x):=u(t,x+W_{t})$. That is, if $u$ is
a solution to \eqref{eq:intro-stoch-burgers-homogeneous} then $v$
is a solution to \eqref{eq:intro-det-burgers-homogeneous} and vice
versa. In particular, shocks and non-uniqueness of weak solutions
still appear in \eqref{eq:intro-stoch-burgers-homogeneous}. Hence,
no well-posedness by noise, nor regularization by noise seems to be
present in this case and it was concluded in \cite{FGP10}: \textit{The
generalization to nonlinear transport equations, where $b$ depends
on $u$ itself, would be a major next step for applications to fluid
dynamics but it turns out to be a difficult problem }(cf.~ \cite[p.6, l.11 ff.]{FGP10})\textit{.}

The purpose of this work is to shed more light on the effect of linear
multiplicative noise on nonlinear scalar conservation laws. In contrast
to the above observation, we show that a similar effect of well-posedness
by noise as obtained in \cite{FGP10} for \eqref{eq:intro-det-transport}
can be observed for (nonlinear) scalar conservation laws. More precisely,
we consider inhomogeneous scalar conservation laws with irregular
flux of the type
\begin{align}
 & \partial_{t}u(t,x)+b(x,u(t,x))\cdot\nabla u(t,x)=0\quad\text{on }[0,T]\times\R^{d},\label{eq:intro-det-burgers-inhomo}\\
 & u(0,\cdot)=u_{0},\nonumber 
\end{align}
with a possibly space-irregular $b$. In particular, this includes
the special case of inhomogeneous Burgers' equations $b(x,u)=2b(x)u$.
The model example 
\begin{equation}
b(x,u)=2\sgn(x)(\sqrt{|x|}\wedge K)u\label{eq:model_ex}
\end{equation}
for some $K>0$, $u_{0}(\cdot)=1_{[0,1]}(\cdot)$, $d=1$ shows that
entropy solutions to \eqref{eq:intro-det-burgers-inhomo} are not
necessarily unique. Indeed, fix some time $T>0$ and choose $K>T/2+1$
for simplicity. Then, there are several entropy solutions to \eqref{eq:intro-det-burgers-inhomo},
including the following two particular ones
\[
u^{1}(t,x):=\begin{cases}
1 & \text{if }0\le x\le\left(\frac{t}{2}+1\right)^{2}\\
0 & \text{otherwise}
\end{cases},\quad u^{2}(t,x):=\begin{cases}
1 & \text{if }-\left(\frac{t}{2}\right)^{2}\le x\le\left(\frac{t}{2}+1\right)^{2}\\
0 & \text{otherwise}
\end{cases},
\]
on $[0,T]\times\R$. In contrast, the main result of this work (Theorem
\ref{thm:well_posedness}) shows that entropy solutions to the stochastically
perturbed scalar conservation law \eqref{eq:intro-stoch-burgers-inhomo}
are unique, under certain assumptions on $b$ and its divergence.%
{} Note that \eqref{eq:model_ex} satisfies these assumptions. Hence,
this demonstrates that linear multiplicative noise has a similar regularizing
effect in the case of nonlinear scalar conservation laws with irregular
flux as it was obtained in the linear setting (i.e.~for linear transport
equation \eqref{eq:intro-stoch-transport}) in \cite{FGP10}. To the
authors' knowledge, this provides the first example of a nonlinear
scalar conservation law that becomes well-posed by the inclusion of
noise.

Let us comment more on the role of the noise in the presence of spatial
inhomogeneities. In this setting, the transformation $v(t,x):=u(t,x+W_{t})$,
applied to \eqref{eq:intro-stoch-burgers-homogeneous} above, gives
\begin{align*}
\partial_{t}v(t,x)+b(x+W_{t},v(t,x))\cdot\nabla v(t,x)=0\quad\text{on }[0,T]\times\R^{d}.
\end{align*}
Hence, the noise induces a random shift of the flux function in its
spatial variable. As for the Burgers' equation \eqref{eq:intro-stoch-burgers-homogeneous},
and contrary to the linear transport equation \eqref{eq:intro-stoch-transport}
this random shift cannot prevent the occurrence of shocks. However,
our main result indicates that, as in the linear setting, this random
shift still has an averaging and thus regularizing effect on spatial
inhomogeneities.

Scalar conservation laws with irregular flux in divergence form have
been used in several models, including models of traffic flow, flow
in porous media and sedimentation processes (cf.~Crasta, De Cicco,
De Philippis \cite{CDCDP15}). In the present work, we choose to consider
the non-divergence form in order to allow comparison to the results
obtained in \cite{FGP10} and by Beck, Flandoli, Gubinelli, Maurelli
in \cite{BFGM14}. We expect that related arguments can be also applied
to the corresponding divergence type equations, as it was demonstrated
in the linear setting in \cite{BFGM14}, although nontrivial differences
with the non-divergence case may arise (see Remark \ref{rem:div_form}).
This will be treated in a subsequent work. The respective study of
conservation laws with irregular flux has attracted considerable interest
in recent years, see Andreianov, Karlsen, Risebro \cite{AKR10,AKR11},
Crasta, De Cicco, De Philippis, Ghiraldin \cite{CDCDPG16,CDCDP15},
Andreianov, Mitrovi\'{c} \cite{AD15} among many more. Due to the
spatial irregularity of the flux, entropy solutions to \eqref{eq:intro-det-burgers-inhomo}
are typically non-unique and several selection criteria to select
a unique entropy solution have been introduced, corresponding to different
physical phenomena and relative approximation procedures. Therefore,
the study of selection methods for \eqref{eq:intro-det-burgers-inhomo}
is of high interest. The well-posedness result for
\begin{align}
 & du(t,x)+b(x,u(t,x))\cdot\nabla u(t,x)\,dt+\s\nabla u(t,x)\circ dW_{t}=0\quad\text{on }[0,T]\times\R^{d},\label{eq:intro-stoch-burgers-3}\\
 & u(0,\cdot)=u_{0},\nonumber 
\end{align}
with $\s>0$ obtained in this paper opens the way to study selection
principles by vanishing noise $\s\to0.$ In the case of linear transport
equations with irregular drift such vanishing noise selection methods
have been analyzed by Attanasio, Flandoli \cite{AF09}, Delarue, Flandoli
\cite{AF09,FD14} and it should be noted that in general the vanishing
viscosity selection does not coincide with the vanishing noise selection.
In analogy to linear stochastic transport equations \eqref{eq:intro-stoch-transport},
stochastic scalar conservation laws \eqref{eq:intro-stoch-burgers-inhomo}
model the evolution of passive scalars in turbulent fluids, so-called
Kraichnan models.

The literature on regularization (i.e.~improvement of regularity)
and well-posedness (i.e.~existence, uniqueness and possibly stability)
by noise is vast and giving a complete survey at this point would
exceed the purpose of this paper. Therefore, we will restrict to those
that seem most relevant for the content of this work and refer to
Flandoli \cite{F11}, Flandoli, Romito \cite{FR02,FR08}, Gyöngy,
Pardoux \cite{GP93}, Butkovsky, Mytnik, Leonid \cite{BM16} for further
references and a more complete account of the literature. Concerning
the case of transport equations with irregular drift \eqref{eq:intro-det-transport},
we mention the works by Flandoli, Gubinelli, Priola \cite{FGP10,FGP13},
Flandoli, Fedrizzi \cite{FF13}, Beck, Flandoli, Gubinelli, Maurelli
\cite{BFGM14} and the references therein. In particular, we would
like to emphasize the work \cite{AF11} by Attanasio, Flandoli which
provides a purely analytic approach to the effect of well-posedness
by noise for \eqref{eq:intro-stoch-transport}, since the proof has
served as an inspiration for some of the steps of the proof presented
in this paper. A regularization by noise effect for \eqref{eq:intro-stoch-transport}
has been first obtained in \cite{FF13} where it has been shown that
solutions to \eqref{eq:intro-stoch-transport} are smooth if the initial
condition is, assuming that $b$ satisfies certain integrability conditions,
slightly more restrictive than the Ladyzhenskaya-Prodi-Serrin condition.
A PDE-based approach and a generalization of these results to drifts
$b$ satisfying the Ladyzhenskaya-Prodi-Serrin condition and to divergence-type
equations has been given in \cite{BFGM14}. A path-by-path approach
to well-posedness by noise has been introduced by Catellier, Gubinelli
in \cite{CG16} and was used by Catellier in \cite{Cat2016} for transport
equations. Another approach based on Malliavin calculus has been introduced
by Menoukeu-Pamen, Meyer-Brandis, Nilssen, Proske in \cite{MMNPZ2013}
and developed in a series of papers, cf.~e.g.~Mohammed, Nilssen,
Proske \cite{MohNilPro2015} on transport equations. We also refer
to Duboscq, Réveillac \cite{DR16} for a generalization of \cite{FGP10}
to SDE with random drift.\\
In some (typically nonlinear) situations, the spatial dependence of
the noise coefficients has proven to be crucial in order to obtain
well-posedness by noise results. More precisely, in Flandoli, Gubinelli,
Priola \cite{FGP11} well-posedness by spatially dependent linear
transport noise for point vortex dynamics informally related to stochastic
2D-Euler equations has been shown. In \cite{DFV14} it has been shown
that the same type of noise can prevent the collapse of point charges
in Vlasov-Poisson equations. \\
More recently, regularizing effects of nonlinear noise in the setting
of (nonlinear) scalar conservation laws has been observed by Gess,
Souganidis in \cite{GS14-2} and in the setting of fully nonlinear
PDE by Gassiat, Gess in \cite{GG16}. Well-posedness of stochastic
scalar conservation laws with random flux has been considered by Lions,
Perthame, Souganidis \cite{LPS13,LPS14}, Gess, Souganidis, \cite{GS14},
Mariani \cite{M10}. 

We next present the idea and an outline of the proof. Our treatment
of \eqref{eq:intro-stoch-burgers-inhomo} is based on the kinetic
formulation of (stochastic) scalar conservation laws as introduced
by Lions, Perthame, Tadmor in \cite{LPT94}. For a function $u:[0,T]\times\R^{d}\to\R$
we introduce the kinetic function $\chi(t,x,\xi):[0,T]\times\R^{d}\times\R\to\R$
by
\begin{equation}
\chi(t,x,\xi)=\chi(u(t,x),\xi):=1_{\xi<u(t,x)}-1_{\xi<0}.\label{def chi-1}
\end{equation}

In the case of a smooth spatial inhomogeneity $b$ and smooth driving
signal $W$, $u$ is an entropy solution to \eqref{eq:intro-stoch-burgers-inhomo}
iff $\chi$ solves the following equation,in the sense of distributions,
\begin{align}
\partial_{t}\chi & =-b(x,\xi)\cdot\nabla\chi-\nabla\chi\cdot\dot{W}_{t}+\partial_{\xi}m,\label{eq:intro-kinetic-eqn}
\end{align}
where $m$ is a nonnegative bounded random measure on $[0,T]\times\R^{d}\times\R$
and the derivatives are intended with respect to $x$ unless differently
specified. In the general case of \eqref{eq:intro-stoch-burgers-inhomo},
we take \eqref{def chi-1}, \eqref{eq:intro-kinetic-eqn} as the definition
of an entropy solution to \eqref{eq:intro-stoch-burgers-inhomo},
where now the term $\nabla\chi\cdot\dot{W}_{t}$ should be interpreted
as a Stratonovich integral, or more precisely,
\begin{align}
\partial_{t}\chi & =-b(x,\xi)\cdot\nabla\chi-\nabla\chi\circ dW_{t}+\partial_{\xi}m\label{eq:intro-kinetic-eqn-stoch}\\
 & =-b(x,\xi)\,\nabla\chi-\nabla\chi\cdot dW_{t}+\frac{1}{2}\Delta\chi+\partial_{\xi}m\,,\nonumber 
\end{align}
see Definition \ref{def:path_e-soln-1} below for details. As in the
deterministic case, the notion of a generalized kinetic solution is
convenient in the construction of an entropy solution since, roughly
speaking, the class of generalized kinetic solutions is stable under
weak limits. Roughly speaking, a function $f$ is said to be a generalized
kinetic solution to \eqref{eq:intro-stoch-burgers-inhomo} if $f$
solves \eqref{eq:intro-kinetic-eqn} for some nonnegative measure
$m$ and $|f|=\sgn(\xi)f\le1$, $\partial_{\xi}f=\d_{0}-\nu$ for
some nonnegative measure $\nu$. The key difference to an entropy
solution is that $f$ is not assumed to be of the form of an kinetic
function \eqref{def chi-1} for some function $u$. \\
The main difficulty then lies in proving that generalized kinetic
solutions are in fact entropy solutions, which boils down to proving
$|f|=1$ a.e.. In order to prove this we aim to estimate the difference
$|f|-f^{2}$ based on \eqref{eq:intro-kinetic-eqn-stoch}. The proof
now consists of two steps. In the first step, an (in)equality for
$|f|-f^{2}$ is derived based on renormalization techniques (cf.~\cite{DPL89,A04-1})
using the assumption $b\in L_{\xi,loc}^{1}(W_{x,loc}^{1,1})$. Informally,
this leads to the equality
\[
\partial_{t}(|f|-f^{2})+b(x,\xi)\cdot\nabla(|f|-f^{2})+\nabla(|f|-f^{2})\circ dW_{t}=(\sgn(\xi)-2f)\partial_{\xi}m.
\]
Passing to the Itô formulation and taking the expectation, we informally
``gain a Laplacian'' similarly to \cite{AF11}. The main difficulty
at this point that is due to the nonlinearity of \eqref{eq:intro-stoch-burgers-inhomo}
is the additional singular term $\partial_{\xi}m$. To handle this
term, in the second step, we integrate in both $\omega$ and $\xi$
(in the Itô formulation), which informally yields
\begin{align*}
 & \partial_{t}\int\E(|f|-f^{2})\,d\xi+\int b(x,\xi)\cdot\nabla\E(|f|-f^{2})\,d\xi+\frac{1}{2}\Delta\int\E(|f|-f^{2})\,d\xi\\
 & =\E\int\vp(\sgn(\xi)-2f)\partial_{\xi}m\,d\xi.
\end{align*}
Since $\partial_{\xi}f=\d_{0}-\nu\le\d_{0}$ this implies
\begin{align*}
\partial_{t}\int\E(|f|-f^{2})\,d\xi+\int b(x,\xi)\cdot\nabla\E(|f|-f^{2})\,d\xi+\frac{1}{2}\Delta\int\E(|f|-f^{2})\,d\xi & \le0
\end{align*}
This is a linear \textit{parabolic} PDE in $\int\E(|f|-f^{2})\,d\xi$,
but, in contrast to the linear setting (i.e.\ for linear transport
equation) in \cite{AF11}, this PDE is not closed, since it involves
both $\int\E(|f|-f^{2})\,d\xi$ and $\int b(x,\xi)\cdot\nabla\E(|f|-f^{2})\,d\xi$.
The rigorous analysis is carried out by passing to the distributional
form. The problem that the above PDE is non-closed then relates to
finding a nonnegative test function $\varphi$, independent of $\xi,$
that satisfies for every $\xi$, 
\[
\partial_{t}\vp+\div(b(x,\xi)\vp)+\D\vp\le C,
\]
for some constant $C>0$. In the analysis of this PDE, we rely on
the boundedness assumption on $b$ and the integrability assumption
on $\div b$. We conclude (Lemma \ref{lem:key_lemma}) that
\begin{align*}
\E\int_{\R^{d}\times\R}(|f_{t}|-f_{t}^{2})\,d\xi dx & \le C\int_{\R^{d}\times\R}(|f_{0}|-f_{0}^{2})\,d\xi dx,
\end{align*}

which implies that $|f|=1$ and ends the proof.
\begin{rem}
\label{rem:div_form}Scalar conservation laws in divergence form
\begin{align*}
 & \partial_{t}u+\div(b(x,u)u)=0,
\end{align*}
with spatially irregular drift $b$, can have different pathological
behaviour than equation \eqref{eq:intro-det-burgers-inhomo}, such
as concentration of mass. For further details we refer to \cite{BFGM14,GS17-2}
and the references therein.
\end{rem}

\subsection{Notation}

We let $(\Omega,\mathcal{A},P)$ be a measurable space, $(\mathcal{F}_{t})_{t}$
be a normal filtration on $(\Omega,\mathcal{A},P)$ (i.e.~$(\mathcal{F}_{t})_{t}$
is right-continuous and $\mathcal{F}_{0}$ contains the null sets
of $(\Omega,\mathcal{A},P)$) and $W=(W_{t})_{t}$ be a $d$-dimensional
Brownian motion on $\Omega$ with respect to the filtration $(\mathcal{F}_{t})_{t}$.
For a $\sigma$-finite measure space $(E,\mathcal{E},\mu)$, we say
that a function $f:E\rightarrow\R$ is measurable if, for every Borel
subset $A$ of $\R$, $f^{-1}(A)$ is in $\mathcal{E}$. Given a
Banach space $V$, we define $L^{0}(E;V)=L^{0}(E,\mathcal{E},\mu;V)$
in two cases:
\begin{enumerate}
\item [(1)]if $V=U^{*}$ is the dual space of a separable Banach space
$U$, $L^{0}(E;V)$ is defined as the space of classes of equivalence,
under the relation ``$f=g$ $\mu$-a.e.'', of weakly-{*} measurable
functions $f:E\rightarrow V$, i.e., for every $\varphi$ in $U$,
$x\mapsto\langle f(x),\varphi\rangle_{V,U}$ is measurable. This applies
to the case of $V=\mathcal{M}(D)$, the space of finite signed measure
over a domain $D$ of $\R^{n}$, $L^{\infty}(D)$, $L^{p}(D)$ for
$1<p<\infty$;
\item [(2)]if $V$ is separable, $L^{0}(E;V)$ is defined as the space
of classes of equivalence, under the relation ``$f=g$ $\mu$-a.e.'',
of weakly measurable functions $f:E\rightarrow V$, i.e., for every
$\varphi$ in $V^{*}$, $x\mapsto\langle f(x),\varphi\rangle_{V,V^{*}}$
is measurable. This applies to the case of $V=C_{0}(D)$, the space
of continuous bounded function on a domain $D$ of $\R^{n}$ vanishing
at infinity, $L^{1}(D)$, $L^{p}(D)$ for $1<p<\infty$.
\end{enumerate}
Similarly, one can define $L^{p}(E;V).$ When $V=\R$ we simply write
$L^{0}(E)$, $L^{p}(E)$ (the usual $L^{p}$ spaces). For a metric,
locally compact, $\sigma$-compact space $S$, the space $\mathcal{M}(S)$
denotes the space of finite signed Borel measures on $S$, $\mathcal{M}_{+}(S)$
the subset of finite nonnegative Borel measures. More details on these
spaces and on measurability issues are given in the appendix. When
not otherwise stated, the spaces $\Omega$, resp. $[0,T]\times\Omega$
are considered endowed with the $\sigma$-algebrae $\mathcal{A}$,
resp. $\mathcal{B}([0,T])\otimes\mathcal{A}$; $\mathcal{P}$ denotes
the progressive $\sigma$-algebra on $[0,T]\times\Omega$, progressive
measurability is measurability with respect to $\mathcal{P}$. We
use progressive measurability instead of predictability because càdlàg
adapted process are $\mathcal{P}$-measurable. The stochastic Fubini
theorem \cite[Exercise 5.17]{RevYor1999}, that we put to use later-on,
is stated for predictable processes but can be immediately extended
to progressively measurable processes in our context (where the martingale
part of the integrator is a Brownian motion or an integral with respect
to Brownian motion). The concepts of entropy solutions, kinetic solutions,
generalized kinetic solutions, kinetic measures have always to be
understood in the sense of equivalence classes, although we will often
consider them as functions when this does not create confusion. In
cases where we need to work with representatives this will be indicated,
although we will often use the same symbol for the class and the representative.

The variables $t$, $\omega$, $x$, $\xi$ denote elements resp.
in $[0,T],$ $\Omega$, $\R^{d}$, $\R$. We often use the short notation
$L_{x}^{p}$, $L_{t,\omega,x}^{p}$, $\mathcal{M}_{x}$, ... for the
spaces $L^{p}(\R^{d}),$ $L^{p}([0,T]\times\Omega\times\R^{d})$,
... and $L_{\xi,[-R,R]}^{p}$ for the space $L^{p}([-R,R])$. We also
use the notation $b\in L_{x}^{p}(L_{\xi,loc}^{\infty})$, $b\in L_{\xi,loc}^{1}(W_{x,loc}^{1,1})$,
... to state that $b\in L_{x}^{p}(L_{\xi,[-R,R]}^{\infty})$ for every
$R>0$, $b\in L_{\xi,[-R,R]}^{1}(W_{x,B_{R}}^{1,1})$ for every $R>0$,
.... The symbols $\nabla$, $\div$, $\Delta$, if not differently
specified, are referred to derivatives in $x$, while derivatives
in $t$ and $\xi$ are denoted by $\partial_{t}$, $\partial_{\xi}$.
As usual in probability theory, $\varphi_{t}$ denotes the evaluation
at time $t$, that is, $\varphi_{t}=\varphi(t)$ {[}note however that
the subscript $t$ in $\partial_{t}$ does only denote the time derivative
and not its evaluation at $t$: in particular, $\int_{0}^{t}g\partial_{t}\varphi\,dr=\int_{0}^{t}g(r)\partial_{t}\varphi(r)\,dr$
denotes the integral from $0$ to $t$ of $g$ times the time derivative
of $\varphi${]}. The symbol $\langle\cdot,\cdot\rangle$ denotes
the scalar product in $L_{x,\xi}^{2}$, unless differently specified.
For example, $\langle\cdot,\cdot\rangle_{t,x,\xi}$ denotes the scalar
product in $L_{t,x,\xi}^{2}$. Sometimes, for a measure $m$ on $[0,T]\times\R^{d}\times\R$,
we use the notation $\langle m,\varphi\rangle dt$ for $\int_{\R^{d}\times\R}\varphi(t,x,\xi)m(dt,dx,d\xi)$.
The convolution operator is denoted by $*_{var}$, where $var$ stands
for the variable (usually $x$ or $\xi$ or both) for which the convolution
is performed. The function $\rho$ denotes a smooth nonnegative compactly
supported even function on $\R^{d}$ such that $\int_{\R^{d}}\rho(x)\,dx=1$,
and $\rho^{\epsilon}:=\epsilon^{-d}\rho(\epsilon^{-1}\cdot)$. Similarly
$\bar{\rho}$ denotes a smooth nonnegative compactly supported even
function on $\R$ such that $\int_{\R}\bar{\rho}(\xi)\,d\xi=1$, and
$\bar{\rho}^{\delta}=\delta^{-1}\bar{\rho}(\delta^{-1}\cdot)$. In
statements and proofs, the letter $C$ denotes a generic positive
constant, which can change from line to line and can depend on $d$
(dimension) and $p$ (integrability exponent assumed for $\div\,b$).
In accordance to \eqref{def chi-1} we use the notation $\chi(\xi,u)=1_{\xi<u}-1_{\xi<0}$.
When we use the kinetic formulation, we write $b$ for $b(x,\xi)$.

\subsection{Organization of the paper}

In Section \ref{sec:Some-general-results} we introduce the notions
of entropy, kinetic and generalized kinetic solutions to \eqref{eq:intro-stoch-burgers-inhomo},
prove a flow-transformation result linking \eqref{eq:intro-stoch-burgers-inhomo}
to a scalar conservation law with random coefficients and prove the
existence of generalized entropy solutions based on stable $L^{p}$-estimates.
Some subtle measurability properties are postponed to the second Appendix
\ref{sec:Appendix}. The results and definitions in Section \ref{sec:Some-general-results}
are applicable under mild assumptions on $b$ and, in particular,
apply without change to the non-perturbed case. In Section \ref{sec:Well-posedness-of-entropy}
it is shown that generalized entropy solutions are entropy solutions
and their uniqueness is deduced using certain parabolic PDE estimates
given in the first Appendix \ref{sec:PDE}. 

\section{Definitions and the existence of generalized Kinetic solutions\label{sec:Some-general-results}}

In this section we give some general definitions and results, which
hold also without noise. In the case of a smooth vector field $b$,
there exists a unique entropy solution. In the general case, even
the existence of an entropy solution may not hold in general. However,
one can get the existence of a so-called generalized kinetic solution.

We start defining the concept of an entropy solution.
\begin{defn}
\label{def:kinetic_meas}A (stochastic) bounded kinetic measure is
a map $m:\Omega\rightarrow\mathcal{M}([0,T]\times\R^{d}\times\R)$,
weakly-{*} measurable, satisfying the following properties:

\begin{enumerate}
\item $m\in L^{\infty}(\Omega;\mathcal{M}([0,T]\times\R^{d}\times\R))$;
\item $m$ is a.s.~non-negative and supported on $[0,T]\times\R^{d}\times[-R,R]$
for some $R>0$ independent of $\omega$;
\item for every $\varphi\in C_{c}^{\infty}([0,T]\times\R^{d}\times\R)$,
the process $(t,\omega)\mapsto\int_{[0,t]\times\R^{d}\times\R}\varphi\,dm$
is progressively measurable.
\end{enumerate}
\end{defn}

Here and in what follows, we can extend definitions and formulations
to test functions $\varphi$ which are not necessarily compactly supported
in the $\xi$ variable, because of the assumption that $m$ is supported
on $[0,T]\times\R^{d}\times[-R,R]$.
\begin{defn}
Let $b\in L_{loc}^{1}(\R^{d+1})$ with $\div\,b$ in $L_{loc}^{1}(\R^{d+1})$
and let $u_{0}\in(L^{1}\cap L^{\infty})(\R^{d})$. An entropy solution
to \eqref{eq:intro-stoch-burgers-inhomo} is a measurable function
$u:[0,T]\times\Omega\times\R^{d}\rightarrow\R$, such that $\chi(t,\omega,x,\xi)=\chi(u(t,\omega,x),\xi)=1_{\xi<u(t,\omega,x)}-1_{\xi<0}$
satisfies the following properties:

\begin{enumerate}
\item $\chi\in L^{\infty}([0,T]\times\Omega;L^{1}(\R^{d}\times\R))$ and
is supported on $[0,T]\times\Omega\times\R^{d}\times[-R,R]$ for some
$R>0$;
\item $\chi$ is a weakly-{*} progressively measurable $L_{x,\xi}^{\infty}$-valued
process;
\item there exists a bounded kinetic measure $m$ such that, for every test-function
$\varphi\in C_{c}^{\infty}([0,T]\times\R^{d}\times\R)$, it holds,
for a.e. $(t,\omega)$,
\begin{align}
\langle\chi_{t},\varphi_{t}\rangle=\langle\chi_{0},\varphi_{0}\rangle & +\int_{0}^{t}\langle\chi,\partial_{t}\varphi+\div(b(x,\xi)\varphi)\rangle\,dr+\int_{0}^{t}\langle\chi,\nabla\varphi\rangle\,dW\nonumber \\
 & +\frac{1}{2}\int_{0}^{t}\langle\chi,\Delta\varphi\rangle\,dr-\int_{[0,t]\times\R^{d}\times\R}\partial_{\xi}\varphi\,dm,\label{eq:kin_form}
\end{align}
with $\chi_{0}(x,\xi)=1_{\xi<u_{0}(x)}-1_{\xi<0}$.
\end{enumerate}
\end{defn}

The function $\chi$ is called a kinetic solution.

The well-known definitions of entropy solutions, kinetic solutions
and kinetic measures in the case of deterministic scalar conservation
laws are recovered in the above definitions by removing the $\omega$
dependence, the progressive measurability assumptions as well as the
second order term and stochastic integral in \eqref{eq:kin_form}.
\begin{rem}
\label{rem:cadlag}(i) For every kinetic solution $\chi$ and test
function $\varphi\in C_{c}^{\infty}([0,T]\times\R^{d}\times\R)$,
$(t,\omega)\mapsto\langle\chi_{t},\varphi_{t}\rangle$ is a semimartingale
admitting a càdlàg version. More precisely, it admits a version which
is the sum of a continuous martingale and a process with $BV$ paths.
Indeed, for every $\varphi$ and every representative of $m$, for
a.e.~$\omega$, the function $t\mapsto\int_{[0,t]\times\R^{d}\times\R}\partial_{\xi}\varphi\,dm$
is of finite variation.

The processes $\int_{[0,t]\times\R^{d}\times\R}\partial_{\xi}\varphi\,dm$
and $\int_{[0,t)\times\R^{d}\times\R}\partial_{\xi}\varphi\,dm$ are
progressively measurable and resp.~càdlàg, càglàd. Moreover, the
(random) times where the paths are discontinuous depend only on $m$
and not on $\varphi$.

(ii) More general, let $\varphi:[0,T]\times\Omega\times\R^{d}\times\R\times\R^{m}\rightarrow\R$
be a measurable bounded function such that: 1) for every $(x,\xi,z)$,
$(t,\omega)\mapsto\varphi(t,\omega,x,\xi,z)$ is progressively measurable;
2) for a.e.~$\omega$, $(t,x,\xi,z)\mapsto\varphi(t,\omega,x,\xi,z)$
is continuous. Then, for every representative of $m$, the maps
\begin{align*}
(t,z,\omega) & \mapsto\int_{[0,t]\times\R^{d}\times\R}\varphi(r,\omega,x,\xi,z)m(r,x,\xi)\,drdxd\xi,\\
(t,z,\omega) & \mapsto\int_{[0,t)\times\R^{d}\times\R}\varphi(r,\omega,x,\xi,z)m(r,x,\xi)\,drdxd\xi
\end{align*}
are measurable and:\\
1) for each $z$ fixed, progressively measurable in $(t,\omega)$;\\
2) for a.e.~$\omega$, with zero set independent of $z$, and each
$z$ fixed, càdlàg, resp.~càglàd, in $t$;\\
3) for a.e.~$\omega$, with zero set independent of $t$, and each
$t$ fixed, continuous in $z$.

This result follows from the combination of the following two facts:
a) We can apply Remark \ref{rmk:measurability_omega} (ii) below to
$z$ fixed to get the progressive measurability and càdlàg/càglàd
property, the latter depending only on $m$ and not on $z$. b) The
integral is continuous with respect to $z$, as consequence of the
dominated convergence theorem.
\end{rem}

\begin{rem}
By equation \eqref{eq:kin_form}, for every test function $\varphi\in C_{c}^{\infty}([0,T]\times\R^{d}\times\R)$,
the quadratic covariance between $\langle\chi,\nabla\varphi\rangle$
and $W$ is $[\langle\chi,\nabla\varphi\rangle,W]_{t}=\int_{0}^{t}\langle\chi,\D\varphi\rangle\,dr$
for a.e.~$(t,\omega)$. Note that by abuse of notation we here use
$\langle\chi,\nabla\varphi\rangle$ to also denote its càdlàg version.

Therefore, the Stratonovich integral $\int_{0}^{t}\langle\chi,\nabla\varphi\rangle\,\circ dW$
makes sense and equation \eqref{eq:kin_form} can be rewritten as
\begin{align*}
\langle\chi_{t},\varphi_{t}\rangle=\langle\chi_{0},\varphi_{0}\rangle & +\int_{0}^{t}\langle\chi,\partial_{t}\varphi+\div(b\varphi)\rangle\,dr+\int_{0}^{t}\langle\chi,\nabla\varphi\rangle\,\circ dW\\
 & -\int_{[0,t]\times\R^{d}\times\R}\partial_{\xi}\varphi\,dm.
\end{align*}
In particular, we see here that equation \eqref{eq:kin_form} is of
hyperbolic type.
\end{rem}

\begin{rem}
\label{rmk:meas_u_chi}By the definition of $\chi$ we have immediately
that, for every $1\le p<\infty$, for a.e. $(t,\omega,x)$,
\begin{align}
u(t,\omega,x) & =\int_{\R}\chi(t,\omega,x,\xi)\,d\xi,\nonumber \\
\frac{1}{p}|u(t,\omega,x)|^{p} & =\int_{\R}|\xi|^{p-1}\sgn(\xi)\chi(t,\omega,x,\xi)\,d\xi.\label{eq:rel_u_chi}
\end{align}
Therefore, the weakly-{*} progressive measurability of $\chi$ implies
that of $u$ and $|u|^{p}$.

Conversely, if $u$ is an $L_{x}^{\infty}$-valued weakly-{*} progressively
measurable process, then by Proposition \ref{prop:Lp_two_var} below
$u$ is $\mathcal{\mathcal{P\otimes}B}(\R^{d})$-measurable as a real-valued
function of $(t,\omega,x)$ (more precisely, there exists a version
of $u$ which is $\mathcal{\mathcal{P\otimes}B}(\R^{d})$-measurable).
Since $(v,\xi)\mapsto1_{\xi<v}-1_{\xi<0}$ is Borel measurable, the
function $(t,\omega,x,\xi)\mapsto\chi(t,\omega,x,\xi)$ is $\mathcal{\mathcal{P\otimes}B}(\R^{d})\otimes\mathcal{B}(\R)$-measurable;
that is, $\chi$ is a $L_{x,\xi}^{\infty}$-valued weakly-{*} progressively
measurable process.

From the formulas above and the fact that $\chi=0$ for $|\xi|>R$,
it follows that $u$ is in $L^{\infty}([0,T]\times\Omega;L^{\infty}(\R^{d}))\cap L^{\infty}([0,T]\times\Omega;L^{1}(\R^{d}))$.
Hence, $u$ is in $L^{\infty}([0,T]\times\Omega;L^{p}(\R^{d}))$ and
$\chi$ is in $L^{\infty}([0,T]\times\Omega;L^{p}(\R^{d}\times\R))$
for every $p\in[1,\infty]$.
\end{rem}

\subsection{A flow transformation}

Before giving the existence result, we recall the following transformation
that links equation \eqref{eq:intro-stoch-burgers-inhomo} to a scalar
conservation law with random coefficients.
\begin{prop}
\label{prop:transformation}Let $b\in L_{loc}^{1}(\R^{d+1})$ with
$\div\,b$ in $L_{loc}^{1}(\R^{d+1})$. A function $u$ is an entropy
solution to \eqref{eq:intro-stoch-burgers-inhomo} iff the function
$\td u(t,x):=u(t,x+W_{t})$ is $L_{x}^{\infty}$-valued weakly-{*}
progressively measurable and is a.s.~an entropy solution to 
\begin{equation}
\partial_{t}\td u(t,x)+b(x+W_{t},\td u(t,x))\cdot\nabla\td u(t,x)=0.\label{eq:transformed}
\end{equation}
More precisely, $\chi=\chi(u)$ is a kinetic solution to \eqref{eq:intro-stoch-burgers-inhomo}
with kinetic measure $m$ iff:

\begin{enumerate}
\item $\td\chi(t,x,\xi):=1_{\xi<\td u(t,\omega,x)}-1_{\xi<0}=\chi(t,x+W_{t},\xi)$
is $L_{x,\xi}^{\infty}$-valued weakly-{*} progressively measurable.
\item $\td m(t,x,\xi)=m(t,x+W_{t},\xi)$ is weakly-{*} progressively measurable,
that is, for every $\psi\in C_{c}^{\infty}([0,T]\times\R^{d}\times\R)$,
the process $(t,\omega)\mapsto\int_{[0,t]\times\R^{d}\times\R}\psi\,d\td m$
is progressively measurable.
\item For a.e.~$\omega$, $\td\chi{}^{\omega}$ is a kinetic solution to
\eqref{eq:transformed} with kinetic measure $\td m{}^{\omega}$.
In particular, in the sense of distributions,
\begin{align}
\partial_{t}\td\chi & +b(x+W_{t},\xi)\cdot\nabla\td\chi=\partial_{\xi}\td m.\label{eq:transformed_kinetic}
\end{align}
\end{enumerate}
\end{prop}

\begin{proof}
\textit{Step 1:} Progressive measurability.

Progressive measurability of $\td\chi$ can be deduced from progressive
measurability of $\chi$ and vice versa. Indeed, for every $\varphi$
in $C_{c}^{\infty}(\R^{d}\times\R)$, $\langle\td\chi,\tilde{\varphi}\rangle=\langle\chi,\tilde{\varphi}(x-W_{t},\xi)\rangle$
is progressively measurable, by Remark \ref{rmk:measurability_omega}.
A similar reasoning applies to $\td m$: For every $\tilde{\varphi}\in C_{c}^{\infty}([0,T]\times\R^{d}\times\R)$,
the process $(t,\omega)\mapsto\int_{[0,t]\times\R^{d}\times\R}\tilde{\varphi}\,d\td m=\int_{[0,t]\times\R^{d}\times\R}\varphi(r,x-W_{r},\xi)\,dm$
is progressively measurable, again by Remark \ref{rmk:measurability_omega}.

\textit{Step 2:} Equation \eqref{eq:kin_form} implies \eqref{eq:transformed_kinetic}.

Since $\langle\td\chi,\tilde{\varphi}\rangle=\langle\chi,\tilde{\varphi}(x-W_{t},\xi)\rangle$
for any (deterministic) test function $\tilde{\varphi}$, the statement
would follow if we could take $\tilde{\varphi}(x-W_{t})$ as a test
function. Unfortunately, this is not possible, since $\tilde{\varphi}(x+W_{t})$
is not deterministic. Therefore we use a regularization procedure:
We consider $\chi^{\epsilon}$, a regularization of $\chi$ with respect
to $x$ and $\xi$. Then, for fixed $x$ and $\xi$, we multiply $\chi^{\epsilon}$
by $\tilde{\varphi}(x-W_{t})$ using Itô's formula, integrate in $x$
and $\xi$ and pass to the limit $\epsilon\rightarrow0$.

We consider a regularization of $\chi$ in both $x$ and $\xi$, i.e.~$\chi_{t}^{\epsilon}(x,\xi):=\langle\chi_{t},\rho_{\epsilon}(x-\cdot)\bar{\rho}_{\epsilon}(\xi-\cdot)\rangle$.
For every $(x,\xi)$, we have the following equation, outside a null
set possibly depending on $(x,\xi)$ and $\epsilon$:
\begin{align*}
\chi^{\epsilon}(t,x,\xi) & =\chi^{\epsilon}(0,x,\xi)+\int_{0}^{t}\langle\chi,\div(b(\cdot,\cdot)\rho_{\epsilon}(x-\cdot)\bar{\rho}_{\epsilon}(\xi-\cdot))\rangle\,dr\\
 & \ +\int_{0}^{t}\langle\chi,\nabla\rho_{\epsilon}(x-\cdot)\bar{\rho}_{\epsilon}(\xi-\cdot)\rangle\,dW+\frac{1}{2}\int_{0}^{t}\langle\chi,\Delta\rho_{\epsilon}(x-\cdot)\bar{\rho}_{\epsilon}(\xi-\cdot)\rangle\,dr\\
 & \ -\int_{[0,t]}\langle m,\partial_{\xi}\rho_{\epsilon}(x-\cdot)\bar{\rho}_{\epsilon}(\xi-\cdot)\rangle\,dr.
\end{align*}
We multiply $\chi^{\epsilon}$ by $\tilde{\varphi}(t,x-W_{t},\xi)$
and use Itô's formula for càdlàg processes (see, for example, \cite[Chapter II Theorem 33]{Pro2004}),
applied to $f(x,y)=xy$. Note that no jump term appears here because
the function $f$ is bilinear and thus, with the notation of Protter
\cite[Chapter II Theorem 33]{Pro2004}, $f(x_{s},y_{s})-f(x_{s-},y_{s})-\partial_{x}f(x_{s-},y_{s})\cdot\Delta x_{s}=0$.
Hence, we get, outside a null set as above,
\begin{align*}
 & \chi^{\epsilon}(t,x,\xi)\tilde{\varphi}(t,x-W_{t},\xi)\\
 & =\chi^{\epsilon}(0,x,\xi)\tilde{\varphi}(0,x,\xi)+\int_{0}^{t}\chi^{\epsilon}(r,x,\xi)\partial_{t}\tilde{\varphi}(r,x-W,\xi)dr\\
 & +\int_{0}^{t}\langle\chi,\div(b(\cdot,\cdot)\rho_{\epsilon}(x-\cdot)\bar{\rho}_{\epsilon}(\xi-\cdot))\rangle\tilde{\varphi}(r,x-W,\xi)\,dr\\
 & \ +\int_{0}^{t}\langle\chi,\nabla\rho_{\epsilon}(x-\cdot)\bar{\rho}_{\epsilon}(\xi-\cdot)\rangle\tilde{\varphi}(r,x-W,\xi)\,dW\\
 & \ +\frac{1}{2}\int_{0}^{t}\langle\chi,\Delta\rho_{\epsilon}(x-\cdot)\bar{\rho}_{\epsilon}(\xi-\cdot)\rangle\tilde{\varphi}(r,x-W,\xi)\,dr\\
 & \ -\int_{[0,t]}\langle m,\rho_{\epsilon}(x-\cdot)\partial_{\xi}\bar{\rho}_{\epsilon}(\xi-\cdot)\rangle\tilde{\varphi}(r,x-W,\xi)\,dr\\
 & \ -\int_{0}^{t}\chi^{\epsilon}(r,x,\xi)\nabla\tilde{\varphi}(r,x-W,\xi)\,dW+\frac{1}{2}\int_{0}^{t}\chi^{\epsilon}(r,x,\xi)\Delta\tilde{\varphi}(r,x-W,\xi)\,dr\\
 & \ +\int_{0}^{t}\langle\chi,\nabla\rho_{\epsilon}(x-\cdot)\bar{\rho}_{\epsilon}(\xi-\cdot)\rangle\cdot\nabla\tilde{\varphi}(r,x-W,\xi)\,dr.
\end{align*}
By the stochastic Fubini theorem (see for example Revuz, Yor \cite[Exercise 5.17]{RevYor1999})
and Remark \ref{rem:cadlag}, all the addends have measurable versions
in $(t,\omega,x,\xi)$; moreover, for these versions the equality
above is true for~a.e. $(t,\omega,x,\xi)$, we can integrate in $x$
and in $\xi$ and exchange the order of integration. We do so and
bring the convolution on $\varphi$: we get, with $\varphi(t,x,\xi)=\tilde{\varphi}(t,x-W_{t},\xi)$,
\begin{align*}
\langle\chi_{t},\varphi_{t}^{\epsilon}\rangle & =\langle\chi_{0},\varphi_{0}^{\epsilon}\rangle+\int_{0}^{t}\langle\chi,\partial_{t}\varphi^{\ve}\rangle dr+\int_{0}^{t}\langle\chi,\div(b\varphi^{\epsilon})\rangle\,dr\\
 & \quad+\int_{0}^{t}\langle\chi,\nabla\varphi^{\epsilon}\rangle\,dW+\frac{1}{2}\int_{0}^{t}\langle\chi,\Delta\varphi^{\epsilon}\rangle\,dr-\int_{[0,t]}\langle m,\partial_{\xi}\varphi^{\epsilon}\rangle\,dr\\
 & \quad-\int_{0}^{t}\langle\chi,\nabla\varphi^{\epsilon}\rangle\,dW+\frac{1}{2}\int_{0}^{t}\langle\chi,\Delta\varphi^{\epsilon}\rangle\,dr-\int_{0}^{t}\langle\chi,\Delta\varphi^{\epsilon}\rangle\,dr\\
 & =\langle\chi_{0},\varphi_{0}^{\epsilon}\rangle+\int_{0}^{t}\langle\chi,\partial_{t}\varphi^{\ve}\rangle dr+\int_{0}^{t}\langle\chi,\div(b\varphi^{\epsilon})\rangle\,dr-\int_{[0,t]}\langle m,\partial_{\xi}\varphi^{\epsilon}\rangle\,dr.
\end{align*}
 Finally, we let $\epsilon$ go to $0$ and use the change of variable
$\tilde{x}=x-W_{t}$, to obtain
\begin{align}
\langle\tilde{\chi}_{t},\tilde{\varphi}_{t}\rangle & =\langle\tilde{\chi}_{0},\tilde{\varphi}_{0}\rangle+\int_{0}^{t}\langle\tilde{\chi},\partial_{t}\tilde{\varphi}\rangle dr+\int_{0}^{t}\langle\tilde{\chi},\div(\tilde{b}\tilde{\varphi})\rangle\,dr-\int_{[0,t]}\langle\tilde{m},\partial_{\xi}\tilde{\varphi}\rangle\,dr.\label{eq:random_CL_nullset}
\end{align}

This formula is valid for every $\tilde{\varphi}$ smooth test function
(with compact support), on a full measure set in $(t,\omega)$ which
can depend on $\tilde{\varphi}$. To make this set independent of
$\tilde{\varphi}$, we use a density argument. Let $D$ be a countable
dense set in $C_{c}^{\infty}(\R^{d}\times\R)$ and let $F$ be a full
measure set in $(t,\omega)$ satisfying: for every $(t,\omega)$ in
$F$ and for every $\tilde{\varphi}$ in $D$, $\tilde{m}(\omega)$
is a bounded measure and \eqref{eq:random_CL_nullset} holds. Now,
for a given test function $\tilde{\varphi}$, we take a sequence $(\tilde{\varphi}_{n})_{n}$
in $D$ converging to $\tilde{\varphi}$ in $C_{b}^{2}$; we pass
to the limit in \eqref{eq:random_CL_nullset} for $\tilde{\varphi}_{n}$
(using that $\tilde{\chi}$ is bounded for every $(t,\omega)$) and
we get \eqref{eq:random_CL_nullset} for $\tilde{\varphi}$ for every
$(t,\omega)$ in $F$. The proof of the first part is complete.

\textit{Step 3:} Equation \eqref{eq:transformed_kinetic} and weak-{*}
progressive measurability imply \eqref{eq:kin_form}.

Since the strategy is similar to that of the first part, we will only
sketch it. We regularize $\tilde{\chi}$ by convolving it with an
approximate identity, obtaining $\tilde{\chi}^{\epsilon}$. The progressive
measurability hypothesis implies that $\tilde{\chi}^{\epsilon}$ is
an Itô process. Therefore, for every test function $\varphi$, we
can multiply it by $\varphi(t,x+W_{t},\xi)$ and apply Itô's formula.
By Fubini's theorem, the stochastic Fubini theorem and Remark \ref{rem:cadlag}
we can integrate in $x$ and in $\xi$ and exchange the order of integration.
Then we bring the convolution on $\varphi$, let $\epsilon$ go to
$0$ and change variable to get finally \eqref{eq:kin_form}.
\end{proof}

\subsection{The case of smooth coefficients}

In this section we consider the case of a smooth coefficient $b\in C_{c}^{\infty}(\R^{d+1})$
and a smooth initial condition $u_{0}\in C_{c}^{\infty}(\R^{d})$
and derive stable a priori bounds. For simplicity of notation, we
set $R_{0}=\|u_{0}\|_{L^{\infty}}$.
\begin{prop}
\label{prop:existence_smooth_b}Let $u_{0}\in C_{c}^{\infty}(\R^{d})$
and $b\in C_{c}^{\infty}(\R^{d+1})$. Then there is a unique entropy
solution $u$ to \eqref{eq:intro-stoch-burgers-inhomo}. Moreover,
we have
\begin{equation}
\esssup_{\omega\in\Omega}\sup_{t\in[0,T]}\|u(t)\|_{L^{\infty}}\le\|u_{0}\|_{L^{\infty}}\label{eq:apriori_linfty_bound}
\end{equation}
and, for every $p\ge1$ finite,
\begin{align}
 & \esssup_{\omega\in\Omega}\sup_{t\in[0,T]}\|u(t)\|_{L^{p}}^{p}+p(p-1)\int_{[0,T]}\int\int|\xi|^{p-2}md\xi dxdr\label{eq:apriori_lp_bounds}\\
 & \le\|u_{0}\|_{L^{p}}^{p}+p\|u_{0}\|_{L_{x}^{\infty}}^{p-1}\|\div\,b\|_{L^{1}([0,T]\times\R^{d}\times[-R_{0},R_{0}])}.\nonumber 
\end{align}
Moreover, $\chi$ and $m$ are supported a.s. on $[0,T]\times\R^{d}\times[-R_{0},R_{0}]$.
\end{prop}

\begin{proof}
\textit{Step 1:} We start with the equation
\begin{align}
\partial_{t}v+g(t,x,v)\cdot\nabla v & =0\label{eq:det_SCL}
\end{align}
for some $g\in C([0,T];C_{b}^{3}(\R^{d+1}))$, i.e.~three times continuously
differentiable with bounded derivatives, and initial condition $u_{0}\in C_{c}^{\infty}(\R^{d})$.

We first note that, due to the regularity of $g$ one may rewrite
\eqref{eq:det_SCL} in divergence form with a force. Following Kružkov
\cite[Theorem 5 and Section 5]{Kru1970}, there exists a bounded entropy
solution $v=v^{g}$ to \eqref{eq:det_SCL}. This solution can be constructed
by first approximating $g$ by a smooth $g^{\d}$ and then considering
a vanishing viscosity approximation. That is, $v$ can be obtained
as an a.e.~limit of the solutions $v^{\ve,\d}$ to
\[
\partial_{t}v^{\epsilon,\d}+g^{\d}(t,x,v^{\epsilon,\d})\cdot\nabla v^{\epsilon,\d}=\epsilon\Delta v^{\epsilon,\d}.
\]
The maximum principle applied to these equations yields $\|v^{\epsilon,\d}\|_{L_{t,x}^{\infty}}\le\|u_{0}\|_{L_{x}^{\infty}}$.
Passing to the limit, we obtain the bound
\begin{equation}
\|v\|_{L_{t,x}^{\infty}}\le\|u_{0}\|_{L_{x}^{\infty}}.\label{eq:linfty_for_v}
\end{equation}

As a consequence of Lécureux-Mercier \cite[Corollary 2.5 and Theorem 2.6]{Lec2010},
the solution $v^{g}$ is in $C([0,T];L^{1}(\R^{d}))$ and is unique.
Moreover, the map $C([0,T];C_{b}^{3}(\R^{d+1}))\ni g\mapsto v^{g}\in C([0,T];L^{1}(\R^{d}))$
is locally Lipschitz continuous.

Denote by $\chi=\chi^{g}$ the associated kinetic solution. The bound
\eqref{eq:linfty_for_v} implies that, for every $t$ and for a.e.~$x$,
$\chi(t,x,\cdot)=\chi^{g}(t,x,\cdot)$ is supported on $[-\|u_{0}\|_{L_{x}^{\infty}},\|u_{0}\|_{L_{x}^{\infty}}]=[-R_{0},R_{0}]$.
We have also $\|\chi_{t}-\chi_{s}\|_{L_{x,\xi}^{1}}=\|v_{t}-v_{s}\|_{L_{x}^{1}}$
and $\|\chi_{t}^{g^{1}}-\chi_{t}^{g^{2}}\|_{L_{x,\xi}^{1}}=\|v_{t}^{g^{1}}-v_{t}^{g^{2}}\|_{L_{x}^{1}}$.
Consequently, $\chi$ is in $C([0,T];L^{1}(\R^{d}\times\R))$ and
the map $C([0,T];C_{b}^{3}(\R^{d+1}))\ni g\mapsto\chi^{g}\in C([0,T];L^{1}(\R^{d}\times\R))$
is locally Lipschitz continuous. As a consequence, the maps
\begin{align*}
[0,T]\times C([0,T];C_{b}^{3}(\R^{d+1}))\ni(t,g) & \mapsto v_{t}^{g}\in L^{1}(\R^{d})\\{}
[0,T]\times C([0,T];C_{b}^{3}(\R^{d+1}))\ni(t,g) & \mapsto\chi_{t}^{g}\in L^{1}(\R^{d}\times\R)
\end{align*}
are continuous.

The existence of a kinetic measure $m^{g}$ associated to $\chi^{g}$
can be derived as by Dalibard in \cite[Section 2.2]{Dal2006} extended
to the time dependent and non conservative case; that is, for every
$\varphi$ compactly supported we have
\begin{align}
\int_{[0,T]\times\R^{d}\times\R}\partial_{\xi}\varphi\,dm^{g}= & -\langle\chi_{T}^{g},\varphi_{T}\rangle+\langle\chi_{0}^{g},\varphi_{0}\rangle\label{eq:kinetic_approx}\\
 & +\int_{0}^{T}\langle\chi^{g},\partial_{t}\varphi+\div(g(r,x,\xi)\varphi)\rangle\,dr.\nonumber 
\end{align}
Therefore, $m^{g}$ is uniquely determined and supported on $[0,T]\times\R^{d}\times[-R_{0},R_{0}]$
and \eqref{eq:kinetic_approx} is satisfied for all smooth $\vp$
compactly supported in $x$.

In order to obtain the estimate \eqref{eq:apriori_lp_bounds}, we
consider the test functions given by $(\sgn(\xi)|\xi|^{p-1})^{\epsilon}\psi_{1/\epsilon}(x)$;
here $\psi_{1/\epsilon}$ is an increasing sequence of smooth functions,
$[0,1]$-valued, with values $1$ on $B_{1/\epsilon}$, $0$ on $B_{2/\epsilon}^{c}$
and such that $|\nabla\psi_{1/\epsilon}(x)|\le2\epsilon$ for every
$x$ and $(\sgn(\xi)|\xi|^{p-1})^{\epsilon}:=\sgn(\cdot)|\cdot|^{p-1}*_{\xi}\bar{\rho}_{\epsilon}$.
In particular, $\sgn(\xi)(\sgn(\xi)|\xi|^{p-1})^{\epsilon}$ is a
sequence of nonnegative functions converging pointwise on $\R\setminus\{0\}$
to $|\xi|^{p-1}$. Moreover, in the case $p>1$, $\partial_{\xi}(\sgn(\xi)|\xi|^{p-1})^{\epsilon}$
converges pointwise on $\R$ to $(p-1)|\xi|^{p-2}$, with the convention
$|0|^{p-2}=+\infty$ for $p<2$ and $|0|^{0}=1$ for $p=2$. Due to
\eqref{eq:kinetic_approx}, we have
\begin{align*}
 & \langle\chi_{t}^{g},(\sgn(\xi)|\xi|^{p-1})^{\epsilon}\psi_{1/\epsilon}\rangle+\int_{[0,t]\times\R^{d}\times\R}\frac{d}{d\xi}(\sgn(\xi)|\xi|^{p-1})^{\epsilon}\psi_{1/\epsilon}\,dm^{g}\\
 & =\langle\chi_{0}^{g},(\sgn(\xi)|\xi|^{p-1})^{\epsilon}\psi_{1/\epsilon}\rangle+\int_{0}^{t}\langle\chi^{g},\div\,g(r,x,\xi)(\sgn(\xi)|\xi|^{p-1})^{\epsilon}\psi_{1/\epsilon}\rangle\,dr\\
 & +\int_{0}^{t}\langle\chi^{g},g(r,x,\xi)(\sgn(\xi)|\xi|^{p-1})^{\epsilon}\cdot\nabla\psi_{1/\epsilon}\rangle\,dr.
\end{align*}
In the case $p>1,$ we take the $\liminf$ for $\epsilon\rightarrow0$
and, recalling that $\chi_{t}^{g}(\sgn(\xi)|\xi|^{p-1})^{\epsilon}=|\chi_{t}^{g}|\sgn(\xi)(\sgn(\xi)|\xi|^{p-1})^{\epsilon}$,
we apply Fatou's lemma for the second term on the left hand side and
the dominated convergence theorem for the remaining terms: we get
\begin{align*}
 & \langle|\chi_{t}^{g}|,|\xi|^{p-1}\rangle+(p-1)\int_{[0,t]\times\R^{d}\times\R}|\xi|^{p-2}\,dm^{g}\\
 & \le\langle|\chi_{0}^{g}|,|\xi|^{p-1}\rangle+\int_{0}^{t}\langle|\chi^{g}|,|\div\,g||\xi|^{p-1}\rangle\,dr.
\end{align*}
 Recalling \eqref{eq:rel_u_chi} and \eqref{eq:linfty_for_v} we obtain
\begin{align}
 & \|v^{g}\|_{L^{p}}^{p}+p(p-1)\int_{[0,t]\times\R^{d}\times\R}|\xi|^{p-2}\,dm^{g}\nonumber \\
 & \le\|u_{0}\|_{L^{p}}^{p}+p\int_{0}^{t}\int\int(\div\,g)|\xi|^{p-1}1_{|\xi|\le R_{0}}\,dxd\xi dr\label{eq:det_Lp_est}\\
 & \le\|u_{0}\|_{L^{p}}^{p}+p\|u_{0}\|_{L_{x}^{\infty}}^{p-1}\|\div\,g\|_{L^{1}([0,T]\times\R^{d}\times[-R_{0},R_{0}])}\nonumber 
\end{align}
In particular, taking $p=2$, we see that $\|m\|_{\mathcal{M}_{t,x,\xi}}$
is bounded in terms of $u_{0}$ and $\|\div\,g\|_{L^{1}([0,T]\times\R^{d}\times[-R_{0},R_{0}])}$.

In the case $p=1$, proceeding as before we get
\begin{align*}
 & \|v^{g}\|_{L^{1}}+\liminf_{\epsilon\rightarrow0}\int_{[0,t]\times\R^{d}\times\R}\frac{d}{d\xi}(\sgn(\xi))^{\epsilon}\psi_{1/\epsilon}(x)\,dm^{g}\\
 & \le\|u_{0}\|_{L_{x}^{1}}+\|\div\,g\|_{L^{1}([0,T]\times\R^{d}\times[-R_{0},R_{0}])}.
\end{align*}
In particular, recalling again \eqref{eq:rel_u_chi}, this gives the
global $L_{x,\xi}^{1}$ bound
\begin{equation}
\sup_{t\in[0,T]}\|\chi^{g}(t)\|_{L_{x,\xi}^{1}}\le\|u_{0}\|_{L_{x}^{1}}+\|\div\,g\|_{L^{1}([0,T]\times\R^{d}\times[-R_{0},R_{0}])}.\label{eq:det_L1_est}
\end{equation}

\textit{Step 2:} We apply the previous results to $g=\tilde{b}^{\omega}=b(x+W_{t}(\o),u)$
and, by Proposition \ref{prop:transformation}, get the existence
of an entropy solution with the desired estimates. The technical details
are not difficult but not immediate, since we have to pass from a
process with values in a space of functions of $x$ to a measurable
function of $(t,\omega,x)$:

1) The map $(t,\omega)\mapsto\bar{u}(t,\omega)=v_{t}^{\tilde{b}^{\omega}}$
is measurable bounded from $\mathcal{B}([0,T])\otimes\mathcal{F}_{T}$
to $\mathcal{B}(L_{x}^{1})$, for every $T$, since it is the composition
of the measurable map $(t,\omega)\mapsto(t,\tilde{b}^{\omega})$ from
$\mathcal{B}([0,T])\otimes\mathcal{F}_{T}$ to $\mathcal{B}([0,T])\otimes\mathcal{B}(C([0,T];C_{b}^{3}(\R^{d+1})))$
and the continuous map $(t,g^{\omega})\mapsto v_{t}^{g}$. Moreover
the $L_{x}^{1}$-valued process $\bar{u}$ has time-continuous paths.
So $\bar{u}$ is actually measurable bounded from $\mathcal{P}$ (the
progressive $\sigma$-algebra) to $\mathcal{B}(L_{x}^{1})$; in particular,
it is weakly measurable with respect to $\mathcal{P}$. Therefore,
by Proposition \ref{prop:Lp_two_var}, there exists $\td u$ in $L^{1}([0,T]\times\Omega\times\R^{d},\mathcal{P}\otimes\mathcal{B}(\R^{d}))$
version of $\bar{u}$ (in the sense that, for a.e. $(t,\omega)$,
$\bar{u}(t,\omega)=\td u(t,\omega)$). By the $L^{\infty}$ bounds,
$\td u$ is in $L_{t,\omega,x}^{\infty}$.

Similarly the map $(t,\omega)\mapsto\bar{\chi}(t,\omega)=\chi_{t}^{\tilde{b}^{\omega}}$
is measurable bounded from $\mathcal{B}([0,T])\otimes\mathcal{F}_{T}$
to $\mathcal{B}(L_{x,\xi}^{1})$, for every $T$, and time-continuous
and thus weakly measurable with respect to $\mathcal{P}$. Setting
$\td\chi=1_{\td u<\xi}-1_{0<\xi}$, $\td\chi$ is a $\mathcal{P}\otimes\mathcal{B}(\R^{d})\otimes\mathcal{B}(\R)$-measurable
version of $\bar{\chi}$, it is in $L^{1}$ and supported on $[-R_{0},R_{0}]$.

2) The map $\omega\mapsto\bar{m}(\omega)=m^{\tilde{b}^{\omega}}$
is bounded as an $\mathcal{M}_{t,x,\xi}$-valued function, nonnegative
and supported on $[0,T]\times\R^{d}\times[-R_{0},R_{0}]$. Concerning
progressive measurability, for $\psi$ in $C_{c}^{\infty}([0,T]\times\R^{d}\times\R)$,
call $\varphi$ a primitive function of $\psi$. Then $\int_{[0,T]\times\R^{d}\times\R}\psi\,d\bar{m}=\int\partial_{\xi}\varphi\,d\bar{m}$
is $\mathcal{F}_{T}$-measurable for every $T$ by equation \eqref{eq:kinetic_approx}
with $g=\tilde{b}^{\omega}$ (where the right-hand side is $\mathcal{F}_{T}$-measurable),
moreover it has càdlàg paths; hence it is progressively measurable.
By equation \eqref{eq:kinetic_approx}, $\td\chi$ and $\bar{m}$
satisfy equation \eqref{eq:transformed_kinetic}.

3) We call $u(t,\omega,x)=\td u(t,\omega,x-W_{t}(\omega))$, $\chi(t,\omega,x,\xi)=\td\chi(t,\omega,x-W_{t}(\omega),\xi)$
and $m(t,\omega,x,\xi)=\bar{m}(t,\omega,x-W_{t}(\omega),\xi)$: more
precisely, for $\omega$ fixed, we define $m(\omega)$ as the image
measure of $\bar{m}(\omega)$ under $(t,x,\xi)\mapsto(t,x+W_{t}(\omega),\xi)$.
Then, by Proposition \ref{prop:transformation}, $u$ and $\chi$
are also $L_{x,\xi}^{\infty}$-valued weakly-{*} progressively measurable,
$m$ is a kinetic measure and $u$ is an entropy solution of \eqref{eq:intro-stoch-burgers-inhomo},
with kinetic function $\chi$ and kinetic measure $m$.

4) Changing variable $x'=x-W_{t}$ in \eqref{eq:det_Lp_est} and in
\eqref{eq:det_L1_est}, we get the estimates \eqref{eq:apriori_linfty_bound}
and \eqref{eq:apriori_lp_bounds}.
\end{proof}

\subsection{\label{sec:generalized_construction}Existence of generalized kinetic
solutions }

We introduce the notion of a generalized kinetic solution.
\begin{defn}
\label{def:path_e-soln-1}Let $f_{0}\in(L^{1}\cap L^{\infty})(\R^{d}\times\R)$.
A generalized kinetic solution to \eqref{eq:intro-stoch-burgers-inhomo}
is a measurable function $f:[0,T]\times\Omega\times\R^{d}\times\R\rightarrow\R$
with the following properties:

\begin{enumerate}
\item $f\in L^{\infty}([0,T]\times\Omega;L^{1}(\R^{d}\times\R))$ and is
supported on $[0,T]\times\Omega\times\R^{d}\times[-R,R]$ for some
$R>0$;
\item $f$ is a weakly-{*} progressively measurable $L_{x,\xi}^{\infty}$-valued
process;
\item there exists a kinetic bounded measure $m$ on $[0,T]\times\Omega\times\R^{d}\times\R$
satisfying: for every $\varphi\in C_{c}^{\infty}([0,T]\times\R^{d}\times\R)$,
it holds for a.e.~$(t,\omega)$, 
\begin{align}
\langle f_{t},\varphi_{t}\rangle= & \langle f_{0},\varphi_{0}\rangle+\int_{0}^{t}\langle f,\partial_{t}\varphi+\div(b(x,\xi)\varphi)\rangle\,dr+\int_{0}^{t}\langle f,\nabla\varphi\rangle\,dW\label{eq:kinetic_measure}\\
 & +\frac{1}{2}\int_{0}^{t}\langle f,\Delta\varphi\rangle\,dr-\int_{[0,t]\times\R^{d}\times\R}\partial_{\xi}\varphi\,dm.\nonumber 
\end{align}
\item there exists a kinetic bounded measure $\nu$ on $[0,T]\times\Omega\times\R^{d}\times\R$,
which moreover is in $L^{\infty}([0,T]\times\Omega,\mathcal{M}(\R^{d}\times\R))$,
satistying: for every $\varphi\in C_{c}^{\infty}([0,T]\times\R^{d}\times\R)$,
it holds for a.e.~$(t,\omega)$, 
\begin{align}
|f(t,x,\xi)| & =\sgn(\xi)f(t,x,\xi)\le1,\quad\text{for a.e. }(x,\xi),\label{eq:gen_kinetic_measure}\\
\langle f_{t},-\partial_{\xi}\varphi_{t}\rangle & =\int\varphi(t,x,0)\,dx-\int_{\R^{d}\times\R}\varphi_{t}\,d\nu_{t}.\nonumber 
\end{align}
\end{enumerate}
\end{defn}

A formal, short-hand notation for \eqref{eq:kinetic_measure} and
\eqref{eq:gen_kinetic_measure} is
\begin{align*}
 & \partial_{t}f+b(x,\xi)\cdot\nabla f+\nabla f\circ dW_{t}=\partial_{\xi}m,
\end{align*}
and
\begin{align*}
|f|(t,x,\xi) & =\sgn(\xi)f(t,x,\xi)\le1,\\
\frac{\partial f}{\partial\xi} & =\d(\xi)-\nu(t,x,\xi).
\end{align*}

\begin{rem}
Kinetic solutions are a particular type of generalized kinetic solutions.
Indeed, if $f_{0}(x,\xi):=\chi(u_{0}(x),\xi)$ and $\chi$ is a kinetic
solution to \eqref{eq:intro-stoch-burgers-inhomo}, with associated
kinetic measure $m$, then $\chi$ is also a generalized solution
with kinetic measure $m$ and $\nu=\delta_{\xi=u(t,\omega,x)}$.
\end{rem}

The following theorem asserts the existence of a generalized kinetic
solution.
\begin{thm}
\label{thm:generalized_existence}Let $b\in L_{loc}^{1}(\R^{d+1})$
with $\div\,b$ in $L_{\xi,loc}^{1}(L_{x}^{1})$ and $u_{0}\in(L^{1}\cap L^{\infty})(\R^{d})$.
Then there exists a generalized kinetic solution $f$ to \eqref{eq:intro-stoch-burgers-inhomo}
starting from $f_{0}(x,\xi):=\chi(u_{0}(x),\xi)$.
\end{thm}

\begin{proof}
\emph{Step 1: }\textit{\emph{Approximation of $f$ and convergence}}\textit{.}

We introduce smooth approximations: $b^{\ve}\in C_{c}^{\infty}(\R^{d+1})$
of $b$ with $b^{\ve}\to b$ in $L_{x,\xi,loc}^{1}$ and $\div\,b^{\ve}\rightarrow\div\,b$
in $L_{\xi,loc}^{1}(L_{x}^{1})$; $u_{0}^{\ve}\in C_{c}^{\infty}(\R^{d})$
of $u_{0}$ with $\|u_{0}^{\ve}\|_{L^{p}}\le\|u_{0}\|_{L^{p}}$ for
all $p\ge1$ and $u_{0}^{\ve}\to u_{0}$ in $L_{x}^{1}$. We consider
the corresponding unique entropy solution $u^{\ve}$ (see Proposition
\eqref{prop:existence_smooth_b}) to 
\[
\partial_{t}u^{\ve}(t,x)+b^{\ve}(x,u^{\ve}(t,x))\cdot\nabla u^{\ve}(t,x)+\nabla u^{\ve}(t,x)\circ dW_{t}=0;
\]
that is $\chi^{\ve}=\chi(u^{\ve})$ solves
\begin{align}
\partial_{t}\chi^{\ve} & +b^{\ve}(x,\xi)\cdot\nabla\chi^{\ve}+\nabla\chi^{\ve}\circ dW_{t}=\partial_{\xi}m^{\ve}.\label{eq:kinetic_eps}
\end{align}
Since $|\chi^{\ve}|\le1$ and $\chi^{\ve}$ are $\mathcal{P}\otimes\mathcal{B}(\R^{d+1})$-measurable,
the sequence $\chi^{\ve}$ converges weakly-{*}, up to taking a subsequence,
to a limit $f$ in $L^{\infty}([0,T]\times\Omega\times\R^{d}\times\R,\mathcal{P}\otimes\mathcal{B}(\R^{d+1}))$.
In particular, $f$ is weakly-{*} progressively measurable as an $L_{x,\xi}^{\infty}$-valued
process.

Note that the sequence $\chi^{\ve}$, as $\mathcal{B}([0,T])\otimes\mathcal{\mathcal{A}}\otimes\mathcal{B}(\R^{d+1})$-measurable
processes, is also weakly-{*} compact in $L^{\infty}([0,T]\times\Omega\times\R^{d}\times\R,\mathcal{B}([0,T])\otimes\mathcal{\mathcal{A}}\otimes\mathcal{B}(\R^{d+1}))$.
Therefore, up to taking a sub-subsequence, we can assume that $\chi^{\ve}$
converges weakly-{*} to $f$ also in $L^{\infty}([0,T]\times\Omega\times\R^{d}\times\R,\mathcal{B}([0,T])\otimes\mathcal{\mathcal{A}}\otimes\mathcal{B}(\R^{d+1}))$.
In particular, we can allow test functions of the form $F(\omega)\varphi(t,x,\xi)$.

\emph{Step 2: }\textit{\emph{Bounds and support of $f$.}}

Using Proposition \ref{prop:existence_smooth_b} and $\div\,b\in L_{\xi,loc}^{1}(L_{x}^{1})$,
we obtain that $(\chi^{\epsilon})^{+}=\chi^{\epsilon}\vee0$ is uniformly
bounded in $L_{t,\omega}^{\infty}(L_{x,\xi}^{1})$. Therefore, identifying
$\chi^{\epsilon}(x,\xi)$ with $\chi^{\epsilon}(x,\xi)\,dxd\xi$,
$(\chi^{\epsilon})^{+}$ is uniformly bounded in $L_{t,\omega}^{\infty}(\mathcal{M}_{x,\xi,+})$.
By Theorem \ref{thm:dual_Lp}, up to the selection of a subsequence,
$(\chi^{\epsilon})^{+}$ converges weakly-{*} in $L_{t,\omega}^{\infty}(\mathcal{M}_{x,\xi,+})$
to an element $g^{+}\in L_{t,\omega}^{\infty}(\mathcal{M}_{x,\xi,+})$.
Similarly $(\chi^{\epsilon})^{-}=(-\chi^{\epsilon})\vee0$ converges
weakly-{*}, up to the selection of a subsequence, to an element $g^{-}\in L_{t,\omega}^{\infty}(\mathcal{M}_{x,\xi,+})$.
Moreover, we can take the same subsequence for the weakly-{*} convergence
of $\chi^{\epsilon}$ in $L_{t,\omega,x,\xi}^{\infty}$ and of $(\chi^{\epsilon})^{+}$
and $(\chi^{\epsilon})^{-}$ in $L_{t,\omega}^{\infty}(\mathcal{M}_{x,\xi,+})$.
By a density argument (on the test functions), we see that $g:=g^{+}-g^{-}=f$.
In particular, 
\begin{align*}
\|f\|_{L_{t,\omega}^{\infty}(L_{x,\xi}^{1})}=\|g\|_{L_{t,\omega}^{\infty}(\mathcal{M}_{x,\xi})} & \le\|g^{+}\|_{L_{t,\omega}^{\infty}(\mathcal{M}_{x,\xi})}+\|g^{-}\|_{L_{t,\omega}^{\infty}(\mathcal{M}_{x,\xi})}\\
 & \le2\sup_{\epsilon}\|\chi^{\epsilon}\|_{L_{t,\omega}^{\infty}(L_{x,\xi}^{1})}.
\end{align*}

For the support property of $f$, again by Proposition \ref{prop:existence_smooth_b}
the functions $\chi^{\epsilon}$ are concentrated a.s.~on $[0,T]\times\R^{d}\times[-\|u_{0}\|_{\infty},\|u_{0}\|_{\infty}|]$.
Therefore,
\begin{align*}
 & \E[F\langle\chi^{\epsilon},\varphi\rangle_{t,x,\xi}]=0
\end{align*}
for every $\varphi$ in $L^{1}([0,T]\times\R^{d}\times\R)$ with support
outside $[0,T]\times\R^{d}\times[-\|u_{0}\|_{\infty},\|u_{0}\|_{\infty}]$
and every $F$ in $L_{\omega}^{1}$. Passing to the limit in the above
equality, we conclude that $f$ is concentrated a.s.~on $[0,T]\times\R^{d}\times[-\|u_{0}\|_{\infty},\|u_{0}\|_{\infty}]$.

\emph{Step 3: }\textit{\emph{Convergence of $m^{\epsilon}$}}\textit{.}

By Proposition \ref{prop:existence_smooth_b} (applied with $p=2$),
$m^{\ve}$ is a bounded sequence in the space $L^{\infty}(\Omega;\mathcal{M}_{+}([0,T]\times\R^{d}\times\R))$.
Therefore, by Theorem \ref{thm:dual_Lp}, it converges weakly-{*},
up to subsequences, to a limit $m$ in $L^{\infty}(\Omega;\mathcal{M}_{+}([0,T]\times\R^{d}\times\R))$.
The support property of $m$ follows from Proposition \ref{prop:existence_smooth_b}
as for $f$, replacing $L^{1}([0,T]\times\R^{d}\times\R)$ with $C_{0}([0,T]\times\R^{d}\times\R)$.

\textit{\emph{Concerning progessive measurability of $m$, we have
to prove that, for every $\varphi$ in $C_{c}^{\infty}([0,T]\times\R^{d}\times\R)$,
the process $(\langle m,1_{[0,t]}\varphi\rangle)_{t}$ is $\mathcal{P}$-measurable.
By right-continuity of the process and of the filtration, it is enough
to show that, for every $t_{0}$, for every positive integer $n,$
the random variable $\langle m,\varphi^{t_{0},n}\rangle$ is $\mathcal{F}_{t_{0}+1/n}$-measurable,
where $\varphi^{t_{0},n}=\varphi1_{[0,t_{0}]}*_{t}1_{[0,1/n]}$. The
functions $\langle m^{\epsilon},\varphi^{t_{0},n}\rangle$ are $\mathcal{F}_{t_{0}+1/n}$-measurable
and, by continuity of $\varphi^{t_{0},n}$, converge to $\langle m,\varphi^{t_{0},n}\rangle$
weakly-{*} in $L_{\omega}^{\infty}$, in particular weakly in $L_{\omega}^{2}$.
Now the space of $\mathcal{F}_{t_{0}+1/n}$}}-measurable $L_{\omega}^{2}$
functions is (isomorphic to) a closed (and thus weakly closed) subspace
of $L_{\omega}^{2}$. Hence, $\langle m,\varphi^{t_{0},n}\rangle$
is \textit{\emph{$\mathcal{F}_{t_{0}+1/n}$}}-measurable (precisely,
$\mathcal{F}_{t_{0}+1/n}$-measurable up to $P$-null sets which implies
\textit{\emph{$\mathcal{F}_{t_{0}+1/n}$}}-measurability by completeness
of $\mathcal{F}_{0}$). This shows that $m$ is a kinetic measure.

\emph{Step 4:}\textit{\emph{ Equation \eqref{eq:kinetic_measure}}}\textit{.}

Equation \eqref{eq:kinetic_measure} is obtained passing to the limit
in \eqref{eq:kinetic_eps} for $\varphi$ in $C_{c}^{\infty}([0,T]\times\R^{d}\times\R)$,
exploiting the linearity of the equation. More precisely, we multiply
\eqref{eq:kinetic_eps} by a measurable bounded function $G=G(t,\omega)$,
we integrate in $t$ and $\omega$; we can then pass to the limit
as $\epsilon\rightarrow0$, thanks to the weak-{*} convergence of
$\chi^{\epsilon}$ and $m^{\epsilon}$ and the fact that $b\cdot\nabla\varphi$
and $\varphi\div\,b$ are in $L_{t,\omega,x,\xi}^{1}$. By arbitrariness
of $G$ we get \eqref{eq:kinetic_measure}.

\emph{Step 5: }\textit{\emph{Properties \eqref{eq:gen_kinetic_measure}}}\textit{.}

The bound $\|f\|_{L_{t,\omega,x,\xi}^{\infty}}\le1$ follows from
the same bound for $\chi^{\epsilon}$. For the property $|f|=\sgn(\xi)f$,
we notice that $\E[\langle\chi^{\epsilon},\sgn(\xi)G\rangle_{t,x,\xi}]\ge0$
for every $G$ nonnegative function in $L^{1}([0,T]\times\Omega\times\R^{d}\times\R)$
and we pass to the limit as $\epsilon\rightarrow0$, getting $|f|=\sgn(\xi)f$.

Further, for $\nu^{\ve}$, we have for a.e.~$(t,\omega)$,
\begin{align}
\langle\chi_{t}^{\ve},-\partial_{\xi}\varphi_{t}\rangle & =\int_{\R^{d}}\varphi(t,x,0)\,dx-\int_{\R^{d}\times\R}\varphi_{t}\,d\nu_{t}^{\ve},\label{eq:gen_kinetic_eps}
\end{align}
where $d\nu^{\ve}=\delta_{\xi=u^{\ve}(t,x)}dxdt$. In particular,
$\nu^{\ve}$ is a bounded sequence in $L^{\infty}(\Omega\times[0,T];\mathcal{M}(\R^{d}\times\R))$
of kinetic measures. Proceeding as for $m^{\epsilon}$, we get that
$\nu^{\ve}$ converges weakly-{*}, up to subsequences, to a bounded
kinetic measure $\nu$. We then pass to the limit in \eqref{eq:gen_kinetic_eps}
in a way similar to the proof of equation \emph{\eqref{eq:kinetic_measure}}
and we obtain \eqref{eq:gen_kinetic_measure}.
\end{proof}
\begin{rem}
\label{rmk:enlarged_test}For any generalized kinetic solution $f$,
which by definition is in $L^{\infty}(\Omega\times[0,T];L^{1}(\R^{d}\times\R))\cap L^{\infty}(\Omega\times[0,T];L^{\infty}(\R^{d}\times\R))$,
we have by interpolation $f\in L^{\infty}(\Omega\times[0,T];L^{p}(\R^{d}\times\R))$
for every $1\le p\le\infty$. Moreover the global $L_{x,\xi}^{1}$
bound allows to consider also bounded test functions, independent
of $\xi$, which are in $L^{\infty}([0,T];W^{2,\infty}(\R^{d}))\cap L^{\infty}(\R^{d};W^{1,\infty}([0,T]))$.
\end{rem}

The following lemma will be useful in the next section.
\begin{lem}
\label{lem:version_conv} Let $f$ be a generalized kinetic solution
to equation \eqref{eq:intro-stoch-burgers-inhomo}. For every test
function $\psi$ in $C_{c}^{\infty}(\R^{d}\times\R)$, there exist
measurable functions $f(\psi)^{+}$, $f(\psi)^{-}$ on $[0,T]\times\Omega\times\R^{d}\times\R$,
such that:

\begin{enumerate}
\item $f(\psi)^{+}$, $f(\psi)^{-}$ are versions of $f*_{x,\xi}\psi$ (that
is, for every $(x,\xi)$, $f(\psi)^{+}(x,\xi)$ and $f(\psi)^{-}(x,\xi)$
coincide with \textup{$f*_{x,\xi}\psi$} on a full-measure set in
$[0,T]\times\Omega$, possibly depending on $(x,\xi)$ and $\psi$);
\item for every $(x,\xi)$, $f(\psi)^{+}(x,\xi)$, $f(\psi)^{-}(x,\xi)$
are progressively measurable processes;
\item for a.e. $\omega$ it holds: for every $(x,\xi)$, $f(\psi)^{+}(x,\xi)$
is càdlàg, $f(\psi)^{-}(x,\xi)$ is càglàd;
\item for a.e. $\omega$ it holds: for every $t$, $f(\psi)^{+}$ is $C_{x,\xi}^{1}$
and \textup{$\nabla_{x,\xi}f(\psi)^{+}=f(\nabla_{x,\xi}\psi)^{+}$}
and similarly for $f(\psi)^{-}$.\textup{ }
\end{enumerate}
\end{lem}

The above lemma is similar to Remark \ref{rem:cadlag} but with the
additional property (iv). The existence of such versions is needed
when dealing with terms of the form $\int\partial_{\xi}f(\psi)^{+}\,dm$,
since for these terms both the precise version in time and the differentiability
in $\xi$ are needed. Note that properties (ii) and (iv) imply $\mathcal{P}\otimes\mathcal{B}(\R^{d+1})$-measurability.
\begin{proof}
[Proof of Lemma \ref{lem:version_conv}] We call $\varphi^{x,\xi}(y,\zeta)=\psi(x-y,\xi-\zeta)$.
We know that, for every $(x,\xi)$, it holds for a.e. $(t,\omega)$,
\begin{align}
\langle f_{t},\varphi_{t}^{x,\xi}\rangle & =\langle f_{0},\varphi_{0}^{x,\xi}\rangle+\int_{0}^{t}\langle f,\partial_{t}\varphi^{x,\xi}+\div_{y}(b\varphi^{x,\xi})\rangle\,dr+\int_{0}^{t}\langle f,\nabla_{y}\varphi^{x,\xi}\rangle\,dW\nonumber \\
 & \quad+\frac{1}{2}\int_{0}^{t}\langle f,\Delta_{y}\varphi^{x,\xi}\rangle\,dr-\int_{[0,t]\times\R^{d}\times\R}\partial_{\zeta}\varphi^{x,\xi}\,dm.\label{eq:conv_version}
\end{align}

For the integrals $\int_{0}^{t}\langle f,\partial_{t}\varphi^{x,\xi}+\div_{y}(b\varphi^{x,\xi})\rangle\,dr$,
$\frac{1}{2}\int_{0}^{t}\langle f,\Delta_{y}\varphi^{x,\xi}\rangle\,dr$,
there exist resp.\ versions $A(t,\omega,x,\xi)$, $B(t,\omega,x,\xi)$
which satisfy the second and the fourth property above and are continuous
(a.s.) in $(t,x,\xi)$ (these versions are simply the Lebesgue integrals
for a fixed version of $f$). Such a version $C(t,\omega,x,\xi)$
exists also for the stochastic integral $-\int_{0}^{t}\langle f,\nabla_{y}\varphi^{x,\xi}\rangle\,dW$,
by Theorem 10.6 in Kunita \cite[Chapter 1]{Kun1984}. Finally, for
$-\int_{[0,t]\times\R^{d}\times\R}\partial_{\zeta}\varphi^{x,\xi}\,dm$,
by Remark \ref{rem:cadlag} there exist versions $D^{+}(t,\omega,x,\xi)$,
$D^{-}(t,\omega,x,\xi)$ which satisfy the second and the fourth property
above and are resp.\ càdlàg, càglàd for fixed $(x,\xi)$ (these versions
are simply the Lebesgue integrals resp.\ on $[0,t]\times\R^{d}\times\R$
and $[0,t[\times\R^{d}\times\R$ for a fixed version of $f$). Therefore
$f(\psi)^{+}=\langle f_{0},\varphi_{0}^{x,\xi}\rangle+A+B+C+D^{+}$
and $f(\psi)^{-}=\langle f_{0},\varphi_{0}^{x,\xi}\rangle+A+B+C+D^{-}$
are measurable versions of $f*_{x,\xi}\psi$ with the desired properties.
\end{proof}
From now on, when this does not create confusion, the first three
integrals in formula \eqref{eq:conv_version} will denote their continuous
versions. The càdlàg version $D^{+}$of the last integral will be
denoted still by $-\int_{[0,t]\times\R^{d}\times\R}\partial_{\zeta}\varphi^{x,\xi}\,dm$,
while the càglàd version $D^{-}$ by $-\int_{[0,t)\times\R^{d}\times\R}\partial_{\zeta}\varphi^{x,\xi}\,dm$,
coherently with the continuity property in $t$ of the integral on
$[0,t]$.
\begin{rem}
\label{rem:version_eps}Consider $f^{\epsilon,\delta}=f*_{x,\xi}(\rho_{\epsilon}\bar{\rho}_{\delta})$,
where $\rho=\rho(x)$, $\bar{\rho}=\bar{\rho}(\xi)$ are two $C_{c}^{\infty}$
even functions and $\rho_{\epsilon}(x)=\epsilon^{-d}\rho(\epsilon^{-1}x)$,
$\bar{\rho}_{\delta}(\xi)=\delta^{-1}\bar{\rho}(\delta^{-1}\xi)$.
We call $f^{\epsilon,\delta,+}$, $f^{\epsilon,\delta,-}$ the versions
of $f^{\epsilon,\delta}$ as in the previous Lemma. Note that, by
construction, for a.e. $\omega$, it holds for every $(t,x,\xi)$
(with the above convention on the integrals),
\begin{align*}
f_{t}^{\epsilon,\delta,+}(x,\xi) & =f_{0}(x,\xi)+\int_{0}^{t}\frac{1}{2}\Delta f^{\epsilon,\delta}(x,\xi)\,dr-\int_{0}^{t}(b\cdot\nabla f)^{\epsilon,\delta}(x,\xi)\,dr\\
 & -\int_{0}^{t}\nabla f^{\epsilon,\delta}(x,\xi)\,dW_{r}+\int_{[0,t]\times\R^{d}\times\R}\rho^{\epsilon}(x-y)(\bar{\rho}^{\delta})'(\xi-\zeta)m(r,y,\zeta)\,dyd\zeta dr,
\end{align*}
where $(b\cdot\nabla f)^{\epsilon,\delta}=\nabla(\rho^{\epsilon}\bar{\rho}^{\delta})*_{x,\xi}(bf)+(\rho^{\epsilon}\bar{\rho}^{\delta})*_{x,\xi}((\div\,b)f)$.
The integrands $\int_{\R^{d}\times\R}\rho^{\epsilon}(x-y)\bar{\rho}^{\delta}(\xi-\zeta)m(\cdot,y,\zeta)\,dyd\zeta$,
$\int_{\R^{d}\times\R}\rho^{\epsilon}(x-y)(\bar{\rho}^{\delta})'(\xi-\zeta)m(\cdot,y,\zeta)\,dyd\zeta$
will be denoted resp.~by $m^{\epsilon,\delta}(\cdot,x,\xi)$ and
$\partial_{\xi}m^{\epsilon,\delta}(\cdot,x,\xi)$; they are measures
on $[0,T]$ parametrized by $(\omega,x,\xi)$.

Moreover, for every fixed representative of $m$ and every test function
$\psi$, it holds:\\
1) the function $(t,\omega,x,\xi)\mapsto\int_{[0,t]}f(\psi)^{+}(x,\xi)\partial_{\xi}m^{\epsilon,\delta}(r,x,\xi)\,dr$
is measurable in $(t,\omega,x,\xi)$, càdlàg in $t$ and continuous
in $(x,\xi)$ for a.e.~$\omega$;\\
2) for a.e.~$\omega$, we have for every $t\ge0$,
\begin{align*}
 & \int_{\R^{d}\times\R}\int_{[0,t]}f(\psi)^{+}(r,x,\xi)\partial_{\xi}m^{\epsilon,\delta}(r,x,\xi)\,dr\,dxd\xi\\
 & =\int_{[0,t]\times\R^{d}\times\R}\int_{\R^{d}\times\R}f(\psi)^{+}(r,x,\xi)\rho^{\epsilon}(x-y)(\bar{\rho}^{\delta})'(\xi-\zeta)\,dxd\xi\,m(r,y,\zeta)\,dyd\zeta dr\\
 & =-\int_{[0,t]\times\R^{d}\times\R}\partial_{\xi}f(\psi)^{+}(r,x,\xi)m^{\epsilon,\delta}(r,x,\xi)\,dxd\xi dr.
\end{align*}
Indeed, the measurability follows from Remark \ref{rmk:measurability_omega}
below applied at $(x,\xi)$ fixed and from the continuity property
of the integral with respect to $(x,\xi)$ (at $(t,\omega)$ fixed).
The above equality follows from Fubini's theorem, Lemma \ref{lem:version_conv}
and the càdlàg property of the integrals. An analogous property holds
replacing $f(\psi)^{+}$ with $f(\psi)^{+}\varphi$ or $f(\psi)^{-}\varphi$
for regular test functions $\varphi$.
\end{rem}

\section{Well-posedness of entropy solutions\label{sec:Well-posedness-of-entropy}}

In this section we prove the well-posedness by noise result, namely
the existence, uniqueness and stability of entropy solutions:
\begin{thm}
\label{thm:well_posedness}Assume that $b\in L_{\xi,loc}^{\infty}(L_{x}^{\infty})\cap L_{\xi,loc}^{1}(W_{x,loc}^{1,1})$
and that $\div\,b\in L_{\xi,loc}^{1}(L_{x}^{1})\cap L_{x}^{p}(L_{\xi,loc}^{\infty})$
for some $p>d$, $p\le\infty$. For every initial datum $u_{0}$ in
$(L^{1}\cap L^{\infty})(\R^{d})$, there exists a unique entropy solution
$u$ to \eqref{eq:intro-stoch-burgers-inhomo}. Moreover, for every
initial data $u_{0}^{1}$, $u_{0}^{2}$ in $(L^{1}\cap L^{\infty})(\R^{d})$,
the two corresponding entropy solutions $u^{1},$ $u^{2}$ satisfy
\begin{align*}
\E\int|u_{t}^{1}-u_{t}^{2}|\,dx & \le C\int|u_{0}^{1}-u_{0}^{2}|\,dx,
\end{align*}
for a.e. $t\in[0,T]$ and some constant $C>0$, depending only on
$T$, $\|b\|_{L_{\xi,[-M,M]}^{\infty}(L_{x}^{\infty})}$ and $\|\div\,b\|_{L_{x}^{p}(L_{\xi,[-M,M]}^{\infty})}$,
where $M=\max\{\|u_{0}^{1}\|_{L_{x}^{\infty}},\|u_{0}^{2}\|_{L_{x}^{\infty}}\}$.
\end{thm}

\begin{rem}
As it will be clear from the proof, the result can be generalized
to fluxes with $b(x,u)$ replaced by
\begin{align*}
 & \sum_{k=1}^{N}b_{k}(x,u),
\end{align*}
where $b_{k}$ satisfy the assumptions of Theorem \ref{thm:well_posedness}
with integrability exponents $p_{k}>d$ (i.e.~$\div\,b_{k}\in L_{x}^{p_{k}}(L_{\xi,loc}^{\infty})$)
which can vary with $k$.

Another generalization concerns the condition $\div\,b\in L_{\xi,loc}^{1}(L_{x}^{1})\cap L_{x}^{p}(L_{\xi,loc}^{\infty})$,
which can be relaxed to $\div\,b\in L_{x,\xi,loc}^{1}$, $(\div\,b)_{+}\in L_{\xi,loc}^{1}(L_{x}^{1})\cap L_{x}^{p}(L_{\xi,loc}^{\infty})$.
Indeed, only the bound on the positive part of $\div\,b$ is required
for the a priori estimates in the proof of the existence of generalized
solutions as well as in the proof of uniqueness.
\end{rem}

The proof of Theorem \ref{thm:well_posedness} follows from the following
two preliminary results, the key estimate being the following
\begin{lem}
\label{lem:key_lemma}Assume that $b\in L_{\xi,loc}^{\infty}(L_{x}^{\infty})\cap L_{\xi,loc}^{1}(W_{x,loc}^{1,1})$
and that $\div\,b\in L_{x}^{p}(L_{\xi,loc}^{\infty})$ for some $p>d$,
\textup{$p\le\infty$}. Let $f$ be a generalized kinetic solution
to \eqref{eq:intro-stoch-burgers-inhomo}, supported on $[0,T]\times\Omega\times\R^{d}\times[-R,R]$
for some $R\ge0$. Then, 
\begin{align}
\E\int_{\R^{d}\times\R}(|f_{t}|-f_{t}^{2})\,d\xi dx & \le C\int_{\R^{d}\times\R}(|f_{0}|-f_{0}^{2})\,d\xi dx,\label{eq:key_ineq}
\end{align}
for a.e.~$t\in[0,T]$ and some constant $C>0$, depending only on
$T$, $\|b\|_{L_{\xi,[-R,R]}^{\infty}(L_{x}^{\infty})}$ and $\|\div\,b\|_{L_{x}^{p}(L_{\xi,[-R,R]}^{\infty})}$.
\end{lem}

Note that $|f_{t}|-f_{t}^{2}\ge0$ for any generalized kinetic solution,
since $|f|\le1$ by definition. When the initial datum $f_{0}$ is
the kinetic function of some $u_{0}$, that is, if $f_{0}(x,\xi)=\chi(u_{0})(x,\xi)$,
then Lemma \ref{lem:key_lemma} implies that $f$ takes values in
$\{0,\pm1\}$. In this case $f$ is a true kinetic function:
\begin{prop}
\label{prop:reconstruction}Assume that $b$ satisfies the assumptions
of Lemma \ref{lem:key_lemma} and let $f$ be a generalized kinetic
solution to \eqref{eq:intro-stoch-burgers-inhomo} starting from $f_{0}=\chi(u_{0})$,
for some $u_{0}$ in $(L^{1}\cap L^{\infty})(\R^{d})$. Then there
exists an entropy solution $u$ to \eqref{eq:intro-stoch-burgers-inhomo}
such that $f(x,\xi,t)=\chi(\xi,u(x,t))$ a.e. in $(t,\omega,x,\xi)$.
\end{prop}

Lemma \ref{lem:key_lemma} and Proposition \ref{prop:reconstruction},
together with Theorem \ref{thm:generalized_existence}, imply the
well-posedness result Theorem \ref{thm:well_posedness}:
\begin{proof}
[Proof of Theorem \ref{thm:well_posedness}]Concerning the existence
of an entropy solution, Theorem \ref{thm:generalized_existence} yields
the existence of a generalized kinetic solution $f$ to \eqref{eq:intro-stoch-burgers-inhomo}.
Proposition \ref{prop:reconstruction} then implies the existence
of an entropy solution to \eqref{eq:intro-stoch-burgers-inhomo}.

For stability, let $\chi^{i}=\chi(u^{i},\xi)$ be the kinetic functions
associated to $u^{i}$, $i=1,2$. Note that $|\chi^{1}-\chi^{2}|^{2}=|\chi^{1}-\chi^{2}|=1_{u^{1}\le\xi<u^{2}}+1_{u^{2}\le\xi<u^{1}}$
for a.e.~$\xi$ and, in particular, $\int|\chi_{t}^{1}-\chi_{t}^{2}|^{2}\,d\xi=|u_{t}^{1}-u_{t}^{2}|$.
Therefore, the statement is equivalent to
\begin{align}
\E\int|\chi_{t}^{1}-\chi_{t}^{2}|^{2}dxd\xi & \le C\int|\chi_{0}^{1}-\chi_{0}^{2}|^{2}\,dxd\xi.\label{eq:stability_chi-1}
\end{align}

Now consider $f:=\frac{1}{2}(\chi^{1}+\chi^{2})$. Then $f$ is a
generalized kinetic solution, with associated Young measure $\nu=\delta_{0}-\frac{1}{2}(\delta_{\xi=u^{1}}+\delta_{\xi=u^{2}})$.
Moreover,
\begin{align*}
|f|-f^{2} & =\frac{1}{2}\sgn(\xi)(\chi^{1}+\chi^{2})-\frac{1}{4}((\chi^{1})^{2}+(\chi^{2})^{2}+2\chi^{1}\chi^{2})\\
 & =\frac{1}{2}(|\chi^{1}|+|\chi^{2}|)-\frac{1}{4}(|\chi^{1}|+|\chi^{2}|+2\chi^{1}\chi^{2})\\
 & =\frac{1}{4}|\chi^{1}-\chi^{2}|^{2}.
\end{align*}
Therefore, Lemma \ref{lem:key_lemma} implies \eqref{eq:stability_chi-1}.
Uniqueness follows from stability, thus, the proof is complete.
\end{proof}
In order to prove Lemma \ref{lem:key_lemma}, we will use the equations
(more precisely, certain inequalities) satisfied by $|f|$ and $f^{2}$.
We recall that, since $f$ satisfies a transport-type equation, for
any function $\beta$ regular enough, informally $\beta(f)$ also
satisfies a transport-type equation. This property is known as renormalization.
When coming to a rigorous proof, however, problems can appear from
the drift term, when $b$ is not regular enough, and from the kinetic
measure term $m$. The Sobolev assumption on $b$, as in the theory
of DiPerna, Lions \cite{DPL89} and Ambrosio \cite{A04-1}, ensures
that the drift term behaves nicely. The presence of the kinetic measure
$m$ does not allow to write an equation for $|f|$ and $f^{2}$ themselves
but is enough for the following inequality:
\begin{lem}
\label{lem:eq_f2}Assume that $b\in L_{\xi,loc}^{1}(W_{x,loc}^{1,1})$.
Let $f$ be a generalized kinetic solution to \eqref{eq:intro-stoch-burgers-inhomo}.
Then, for every nonnegative test function $\varphi$ in $C_{c}^{\infty}([0,T]\times\R^{d})$
independent of $\xi$, it holds for a.e.~$(t,\omega)$,
\begin{align*}
\int_{\R^{d}\times\R}(|f_{t}|-f_{t}^{2})\varphi_{t}\,dxd\xi\le & \int_{0}^{t}\int_{\R^{d}\times\R}[\partial_{t}\varphi+\frac{1}{2}\Delta\varphi+\div(b\varphi)](|f|-f^{2})\,dxd\xi dr\\
 & +\int_{0}^{t}\int_{\R^{d}\times\R}\nabla\varphi(|f|-f^{2})\,dxd\xi dW_{r}.
\end{align*}
\end{lem}

For the proof of this Lemma we need the following commutator lemma.
Recall that $f^{\epsilon,\delta}=f*_{x,\xi}(\rho^{\epsilon}\bar{\rho}^{\delta})$.
\begin{lem}
\label{lem:commutator}Assume that $b\in L_{\xi,loc}^{1}(W_{x,loc}^{1,1})$.
Then it holds, for every finite $m\ge1$,
\begin{align*}
\lim_{\epsilon\rightarrow0}\lim_{\d\to0} & \E\int_{0}^{T}|\int_{\R^{d}\times\R}\int_{\R^{d}\times\R}f_{r}^{\epsilon,\delta}(x,\xi)f_{r}(y,\zeta)\bar{\rho}_{\delta}(\xi-\zeta)\Big(\nabla\rho_{\epsilon}(x-y)\cdot(b(x,\xi)-b(y,\zeta))\\
 & \qquad+\rho_{\epsilon}(x-y)\div_{y}\,b(y,\zeta)\Big)\varphi_{r}(x)\,dyd\zeta dxd\xi|^{m}\,dr\\
 & =0.
\end{align*}
\end{lem}

\begin{rem}
This is the only point where we need the Sobolev assumption on $b$.
One may note that, in Ambrosio \cite{A04-1}, the renormalization
property for the linear transport equation is proved only assuming
$BV$ regularity for $b$, roughly speaking by showing the above commutator
estimate for a carefully chosen kernel $\rho$. One may expect that
this strategy also works here, but we do not investigate this issue
any further.
\end{rem}

\begin{proof}
[Proof of Lemma \ref{lem:commutator}]The proof is obtained by adapting
the classical commutator lemma (see for example \cite{DPL89,A04-1})
to this anisotropic regularization in $x$ and $\xi$, which was also
used by Chen, Perthame in \cite{ChePer2003}. Since $b$ is weakly
differentiable in the $x$ variable, we have for a.e.~$(x,y,\xi)$
\begin{align*}
b(x,\xi)-b(y,\xi)=\int_{0}^{1}D_{x}b(y+a(x-y),\xi)(x-y)\,da.
\end{align*}
This formula can be obtained by approximation of $b$ in $L_{\xi,loc}^{1}(W_{x,loc}^{1,1})$
with regular $b^{n}$. By the change of variable $z=(x-y)/\epsilon$,
$\eta=(\xi-\zeta)/\delta$, we obtain
\begin{align}
 & \int_{\R^{d}\times\R}\int_{\R^{d}\times\R}f_{r}^{\epsilon,\delta}(x,\xi)f_{r}(y,\zeta)\bar{\rho}_{\delta}(\xi-\zeta)\Big(\nabla\rho_{\epsilon}(x-y)\cdot(b(x,\xi)-b(y,\zeta))\nonumber \\
 & \qquad+\rho_{\epsilon}(x-y)\div_{y}\,b(y,\zeta)\Big)\varphi_{r}(x)\,dyd\zeta dxd\xi\nonumber \\
 & =\int_{0}^{1}\int\bar{\rho}(\eta)\nabla\rho(z)\cdot\int f_{r}^{\epsilon,\delta}(x,\xi)f_{r}(x-\epsilon z,\xi-\delta\eta)\cdot D_{x}b(x-a\epsilon z,\xi)z\varphi_{r}(x)\,dxd\xi dzd\eta da\nonumber \\
 & +\int\bar{\rho}(\eta)\rho(z)\int f_{r}^{\epsilon,\delta}(x,\xi)f_{r}(x-\epsilon z,\xi-\delta\eta)\div\,b(x-\epsilon z,\xi-\delta\eta)\varphi_{r}(x)\,dxd\xi dzd\eta\nonumber \\
 & +\frac{1}{\epsilon}\int\bar{\rho}(\eta)\nabla\rho(z)\cdot\int f_{r}^{\epsilon,\delta}(x,\xi)f_{r}(x-\epsilon z,\xi-\delta\eta)(b(x,\xi-\delta\eta)-b(x,\xi))\varphi_{r}(x)\,dxd\xi dzd\eta\nonumber \\
 & =:A+B+C\label{eq:comm11}
\end{align}
Here and in the following we can suppose without loss of generality
that all the integrals range over a compact set independent of $\epsilon$,
$\delta$, $r$ and $\omega$, since the test functions $\varphi$,
$\rho$, $\bar{\rho}$ are compactly supported and $f_{r}(x,\xi)$
and $f_{r}^{\epsilon,\delta}(x,\xi)$ are compactly supported in $\xi$
uniformly in $\epsilon$, $\delta$, $r$ and $\omega$.

We start with the first integral $A$ on the right hand side of \eqref{eq:comm11}.
We first take the $L_{t,\omega}^{m}$-limit as $\delta\rightarrow0$
and we find that $A$ converges to
\begin{align}
 & \int_{0}^{1}\int\nabla\rho(z)\cdot\int f_{r}^{\epsilon}(x,\xi)f_{r}(x-\epsilon z,\xi)\cdot D_{x}b(x-a\epsilon z,\xi)z\varphi_{r}(x)\,dxd\xi dzda,\label{eq:Aeps}
\end{align}
where $f^{\epsilon}(x,\xi)=f(\cdot,\xi)*\rho^{\epsilon}(x)$. The
proof of this fact is standard and relies on arguments similar to,
but simpler than, those for the limit as $\epsilon\rightarrow0$,
so we omit it. Now we take the $L_{t,\omega}^{m}$-limit of \eqref{eq:Aeps}
as $\epsilon\rightarrow0$. First we fix $z,$ $a$, $r$ and $\omega$.
For the inner integral, we have
\begin{align*}
 & \int f_{r}^{\epsilon}(x,\xi)f_{r}(x-\epsilon z,\xi)D_{x}b(x-a\epsilon z,\xi)z\varphi_{r}(x)\,dxd\xi-\int f_{r}(x,\xi)^{2}\cdot D_{x}b(x,\xi)z\varphi_{r}(x)\,dxd\xi\\
 & =\int f_{r}^{\epsilon}(x,\xi)f_{r}(x-\epsilon z,\xi)(D_{x}b(x-a\epsilon z,\xi)-D_{x}b(x-\epsilon z,\xi))z\varphi_{r}(x)\,dxd\xi\\
 & \quad+\int f_{r}^{\epsilon}(x,\xi)(f_{r}(x-\epsilon z,\xi)D_{x}b(x-\epsilon z,\xi)-f_{r}(x,\xi)D_{x}b(x,\xi))z\varphi_{r}(x)\,dxd\xi\\
 & \quad+\int(f_{r}^{\epsilon}(x,\xi)-f_{r}(x,\xi))f_{r}(x,\xi)D_{x}b(x,\xi)z\varphi_{r}(x)\,dxd\xi
\end{align*}
The first addend on the right hand side above goes to $0$ for $\ve\to0$:
indeed both $D_{x}b(x-a\epsilon z,\xi)$ and $D_{x}b(x-\epsilon z,\xi)$
tend to $D_{x}b(x,\xi)$ in $L_{x,\xi}^{1}$ by continuity of translation
and $f^{\epsilon}(x,\xi)f(x-\epsilon z,\xi)\varphi(x)$ is bounded
in $L_{x,\xi}^{\infty}$ uniformly in $\epsilon$. The second addend
also goes to $0$ for $\ve\to0$: $f(x-\epsilon z,\xi)Db(x-a\epsilon z,\xi)$
tends to $f(x,\xi)D_{x}b(x,\xi)$ in $L_{x,\xi}^{1}$ by continuity
of translation and $f^{\epsilon}(x,\xi)\varphi(x)$ is bounded in
$L_{x,\xi}^{\infty}$ uniformly in $\epsilon$. Finally, the third
addend goes to $0$ by dominated convergence: $f^{\epsilon}(x,\xi)-f(x,\xi)$
tends to $0$ for a.e.~$(x,\xi)$ and the integrand is bounded by
$C|D_{x}b|(x,\xi)$, for some $C>0$. Therefore, for fixed $z$, $\eta$,
$r$ and $a$, the inner integral in the addend $A$ of \eqref{eq:comm11}
converges to $\int f_{r}^{2}(x,\xi)D_{x}b(x,\xi)z\vp_{r}(x)\,dxd\xi$.
Moreover this inner integral is bounded uniformly in $z$, $a$, $r$
and $\omega$, therefore dominated convergence implies for $A$
\begin{align*}
 & \lim_{\ve\to0}\Big(\lim_{\d\to0}\Big(\int_{0}^{1}\int\bar{\rho}(\eta)\nabla\rho(z)\cdot\int f_{r}^{\epsilon,\d}(x,\xi)f_{r}(x-\epsilon z,\xi-\d\eta)\cdot\\
 & \qquad\cdot D_{x}b(x-a\epsilon z,\xi)z\varphi_{r}(x)\,dxd\xi dzd\eta da\Big)\Big)\\
 & =\int_{0}^{1}\int\nabla\rho(z)\cdot\int f_{r}(x,\xi)^{2}\cdot D_{x}b(x,\xi)z\varphi_{r}(x)\,dxd\xi dzda\\
 & =-\int f_{r}(x,\xi)^{2}\div\,b(x,\xi)\varphi_{r}(x)\,dxd\xi,
\end{align*}
where the limits are taken in $L_{t,\omega}^{m}$ and we have used
that $\int\partial_{i}\rho(z)z_{j}\,dz=-\delta_{ij}$.

Similarly, for the second integral $B$ on the right hand side of
\eqref{eq:comm11} we have
\begin{align*}
 & \lim_{\ve\to0}\Big(\lim_{\d\to0}\Big(\int\bar{\rho}(\eta)\rho(z)\int f_{r}^{\epsilon,\d}(x,\xi)f_{r}(x-\epsilon z,\xi-\d\eta)\cdot\\
 & \qquad\cdot\div\,b(x-\epsilon z,\xi-\delta\eta)\varphi_{r}(x)\,dxd\xi dzd\eta\Big)\Big)\\
 & =\int f_{r}(x,\xi)^{2}\div\,b(x,\xi)\varphi_{r}(x)\,dxd\xi,
\end{align*}
where again the limits are taken in $L_{t,\omega}^{m}$. For the third
integral $C$, again with similar reasoning but now taking only the
limit $\delta\rightarrow0$, we get
\begin{align*}
 & \lim_{\delta\rightarrow0}\E\int_{0}^{T}\big|\frac{1}{\epsilon}\int\bar{\rho}(\eta)\nabla\rho(z)\cdot\\
 & \quad\cdot\int f^{\epsilon,\delta}(x,\xi)f(x-\epsilon z,\xi-\delta\eta)(b(x,\xi-\delta\eta)-b(x,\xi))\varphi(x)\,dxd\xi dzd\eta\big|^{m}\,dr=0.
\end{align*}
Putting together these limits we obtain the desired statement.
\end{proof}
Now we prove Lemma \ref{lem:eq_f2}.
\begin{proof}
[Proof of Lemma \ref{lem:eq_f2}]\textit{Step 1:} We start with the
equation for $|f|$. Since, by \eqref{eq:gen_kinetic_measure}, $|f|=f\sgn(\xi)$,
we aim to use $\sgn(\xi)$ as a test function in \eqref{eq:kinetic_measure}.
To do so, we regularize $\sgn$ via $\sgn*_{\xi}\bar{\rho}^{\delta}=:\sgn^{\delta}$.
Note that $\partial_{\xi}\sgn^{\delta}=2\bar{\rho}^{\delta}$. For
technical reasons that will become clear in the second step, we write
an equation for $\int_{\R^{d}\times\R}f^{\epsilon,\delta}\sgn^{\delta}(\xi)\varphi\,dxd\xi$
(where $f^{\epsilon,\delta}=f*_{x,\xi}(\rho^{\epsilon}\bar{\rho}^{\delta})$),
that is, we take $(\sgn^{\delta}(\xi)\varphi)*_{x,\xi}(\rho^{\epsilon}\bar{\rho}^{\delta})$
as a test function in \eqref{eq:kinetic_measure}. Moreover, again
in \eqref{eq:kinetic_measure} we take the càdlàg version of the integral
and thus get, for a.e. $\omega$ (on a full-measure set independent
of $t$), for every $t$,
\begin{align}
 & \int_{\R^{d}\times\R}f_{t}^{\epsilon,\delta,+}\sgn^{\delta}(\xi)\varphi\,dxd\xi\nonumber \\
 & =\int_{\R^{d}\times\R}f_{0}^{\epsilon,\delta}\sgn^{\delta}(\xi)\varphi_{0}\,dxd\xi+\int_{0}^{t}\int f^{\epsilon,\delta}\sgn^{\delta}(\xi)(\partial_{t}\varphi+\frac{1}{2}\Delta\varphi)\,dxd\xi dr\nonumber \\
 & +\int_{0}^{t}\int fb\cdot\nabla(\sgn^{\delta}(\xi)\varphi)^{\epsilon,\delta}\,dxd\xi dr+\int_{0}^{t}\int f\div\,b(\sgn^{\delta}(\xi)\varphi)^{\epsilon,\delta}\,dxd\xi dr\label{eq:first_step}\\
 & -\int_{0}^{t}\int f^{\epsilon,\delta}\sgn^{\delta}(\xi)\nabla\varphi\,dxd\xi dW_{r}-2\int_{[0,t]\times\R^{d}\times\R}\varphi\bar{\rho}^{\delta}(\xi)m^{\epsilon,\delta}\,dxd\xi dr,\nonumber 
\end{align}

where $f^{\epsilon,\delta,+}$ is the càdlàg version of $f^{\epsilon,\delta}$
and $m^{\epsilon,\delta}=m*_{x,\xi}(\rho^{\epsilon}\bar{\rho}^{\delta})$
(see Remark \ref{rem:version_eps}).

\textit{Step 2:} For $f^{2}$, we would like to take $f\varphi$ as
a test function in \eqref{eq:kinetic_measure}. Since $f$ is not
regular, we regularize it in both $x$ and $\xi$. More precisely
we take $f^{\epsilon,\delta,+}$, $f^{\epsilon,\delta,-}$ resp. càdlàg,
càglàd versions of $f^{\epsilon,\delta}$, as in Lemma \ref{lem:version_conv}
and Remark \ref{rem:version_eps}. Itô's formula for càdlàg processes
(cf.~\cite[Chapter II Theorem 33]{Pro2004}) yields
\begin{align*}
 & (f_{t}^{\epsilon,\delta,+}(x,\xi))^{2}\varphi_{t}(x)-(f_{0}^{\epsilon,\delta}(x,\xi))^{2}\varphi_{0}(x)\\
 & =\int_{0}^{t}f_{r}^{\epsilon,\delta,+}(x,\xi)^{2}\partial_{t}\varphi_{r}(x)\,dr+\int_{[0,t]}(f_{r}^{\epsilon,\delta,+}(x,\xi)+f_{r}^{\epsilon,\delta,-}(x,\xi))\varphi_{r}(x)\,df_{r}^{\epsilon,\delta}\\
 & \quad+\int_{0}^{t}\varphi_{r}(x)\,d[f^{\epsilon,\delta}]_{r}\\
 & =\int_{0}^{t}f_{r}^{\epsilon,\delta}(x,\xi)^{2}\partial_{t}\varphi_{r}(x)\,dr+\int_{0}^{t}f_{r}^{\epsilon,\delta}(x,\xi)\Delta f_{r}^{\epsilon,\delta}(x,\xi)\varphi_{r}(x)\,dr\\
 & \quad-\int_{0}^{t}2f_{r}^{\epsilon,\delta}(x,\xi)(b\cdot\nabla f_{r})^{\epsilon,\delta}(x,\xi)\varphi_{r}(x)\,dr-\int_{0}^{t}2f_{r}^{\epsilon,\delta}(x,\xi)\nabla f_{r}^{\epsilon,\delta}(x,\xi)\varphi_{r}(x)\,dW_{r}\\
 & \quad+\int_{[0,t]}(f_{r}^{\epsilon,\delta,+}(x,\xi)+f_{r}^{\epsilon,\delta,-}(x,\xi))\varphi_{r}(x)\partial_{\xi}m^{\ve,\d}(r,x,\xi)\,dxd\xi dr\\
 & \quad+\int_{0}^{t}|\nabla f_{r}^{\epsilon,\delta}(x,\xi)|^{2}\varphi_{r}(x)\,dr.
\end{align*}
This formula is valid for each $(x,\xi)$ for a.e.~$(t,\omega)$,
where the exceptional set may depend on $(x,\xi)$. However, by Remark
\ref{rem:version_eps}, for a fixed representative of $m$, the integral
with the measure $\partial_{\xi}m^{\ve,\d}$ is measurable in $(t,\omega,x,\xi)$,
càdlàg in $t$ for $(\omega,x,\xi)$ fixed and continuous in $(x,\xi)$
for $(t,\omega)$ fixed. Also the other integrals have versions that
are continuous in $(t,x,\xi)$ for $\omega$ fixed and, in particular,
are measurable in $(t,\omega,x,\xi)$. For such versions, for a.e.~$\omega$,
the above equality above holds for every $(t,x,\xi)$.

The idea at this point is first to integrate in $x$ and $\xi$, then
to use integration by parts to bring the derivatives onto $\varphi$
and thereby to get an equation for $f^{\epsilon,\delta,+}$ which
is similar to the one satisfied by $f$ itself, plus a remainder.
Indeed we integrate in $x$ and $\xi$ and use Remark \ref{rem:version_eps},
Fubini's theorem and the stochastic Fubini theorem: we obtain the
following equality, valid for every $t$ and for every $\omega$ in
a full-measure set independent of $t$, 
\begin{align*}
 & \int_{\R^{d}\times\R}(f_{t}^{\epsilon,\delta,+}(x,\xi))^{2}\varphi_{t}(x))\,dxd\xi-\int_{\R^{d}\times\R}(f_{0}^{\epsilon,\delta}(x,\xi))^{2}\varphi_{0}(x))\,dxd\xi\\
 & =\int_{0}^{t}\int_{\R^{d}\times\R}f_{r}^{\epsilon,\delta}(x,\xi)^{2}\left(\partial_{t}\varphi_{r}(x)+\frac{1}{2}\Delta\varphi_{r}(x)+\div(b(x,\xi)\varphi_{r}(x))\right)\,dxd\xi dr\\
 & +2\int_{0}^{t}\int_{\R^{d}\times\R}\int_{\R^{d}\times\R}f_{r}^{\epsilon,\delta}(x,\xi)f_{r}(y,\zeta)\bar{\rho}_{\delta}(\xi-\zeta)\Big(\nabla\rho_{\epsilon}(x-y)\cdot(b(x,\xi)-b(y,\zeta))\\
 & \quad+\rho_{\epsilon}(x-y)\div_{y}\,b(y,\zeta)\Big)\varphi_{r}(x)\,dyd\zeta dxd\xi dr\\
 & -\int_{[0,t]\times\R^{d}\times\R}(\partial_{\xi}f_{r}^{\epsilon,\delta,+}+\partial_{\xi}f_{r}^{\epsilon,\delta,-})(x,\xi)m^{\epsilon,\delta}(r,x,\xi)\varphi_{r}(x)\,dxd\xi dr\\
 & -\int_{0}^{t}\int_{\R^{d}\times\R}f_{r}^{\epsilon,\delta}(x,\xi)^{2}\nabla\varphi_{r}(x)\,dxd\xi dW_{r}.
\end{align*}

For the third addend, note that, for every $(x,\xi)$, it holds, for
a.e. $(t,\omega)$, $\partial_{\xi}f^{\epsilon,\delta}(x,\xi)=\bar{\rho}_{\delta}(\xi)-\nu*_{\xi}\bar{\rho}_{\delta}(\xi)$
(the convolution being in the $\xi$ direction) and so $\partial_{\xi}f^{\epsilon,\delta}(x,\xi)\ge\bar{\rho}_{\delta}(\xi)$.
Therefore, by the càdlàg/càglàd properties of $\partial_{\xi}f^{\epsilon,\delta,+}$
and $\partial_{\xi}f^{\epsilon,\delta,-}$, it holds for a.e.~$\omega$:
for every $(t,x,\xi)$, $\partial_{\xi}f^{\epsilon,\delta,+}(x,\xi)\ge\bar{\rho}_{\delta}(\xi)$
and $\partial_{\xi}f^{\epsilon,\delta,-}(x,\xi)\ge\bar{\rho}_{\delta}(\xi)$.
So we obtain
\begin{align*}
 & -\int_{[0,t]\times\R^{d}\times\R}(\partial_{\xi}f_{r}^{\epsilon,\delta,+}+\partial_{\xi}f_{r}^{\epsilon,\delta,-})(x,\xi)m^{\epsilon,\delta}(r,x,\xi)\varphi_{r}(x)\,dxd\xi dr\\
 & \le-2\int_{[0,t]\times\R^{d}\times\R}\bar{\rho}^{\delta}(\xi)m^{\epsilon,\delta}(r,x,\xi)\varphi_{r}(x)\,dxd\xi dr.
\end{align*}

Here we see the reason for the additional regularization of $f^{\epsilon,\d}$
in the first step: in this way the right hand side of the above inequality
is equal to the last term in formula \eqref{eq:first_step}. In conclusion
we get, for a.e. $\omega$ (on a full-measure set independent of $t$),
for every $t$,
\begin{align}
 & \int_{\R^{d}\times\R}(f_{t}^{\epsilon,\delta,+}\sgn^{\delta}(\xi)-(f_{t}^{\epsilon,\delta,+})^{2})\varphi_{t}\,dxd\xi-\int_{\R^{d}\times\R}(f_{0}^{\epsilon,\delta}\sgn^{\delta}(\xi)-(f_{0}^{\epsilon,\delta})^{2})\varphi_{0}\,dxd\xi\nonumber \\
 & \le\int_{0}^{t}\int f^{\epsilon,\delta}\sgn^{\delta}(\xi)(\partial_{t}\varphi_{r}+\frac{1}{2}\Delta\varphi_{r})\,dxd\xi dr\nonumber \\
 & +\int_{0}^{t}\int fb\cdot\nabla(\sgn^{\delta}(\xi)\varphi))^{\epsilon,\delta}\,dxd\xi dr+\int_{0}^{t}\int f\div\,b(\sgn^{\delta}(\xi)\varphi)^{\epsilon,\delta}\,dxd\xi dr\nonumber \\
 & -\int_{0}^{t}\int_{\R^{d}\times\R}(f^{\epsilon,\delta})^{2}(\partial_{t}\varphi+\frac{1}{2}\Delta\varphi+\div(b\varphi))\,dxd\xi dr\label{eq:ineq_eps}\\
 & -\int_{0}^{t}\int_{\R^{d}\times\R}(f^{\epsilon,\delta}\sgn^{\delta}(\xi)-(f^{\epsilon,\delta})^{2})\nabla\varphi\,dxd\xi dW_{r}\nonumber \\
 & -2\int_{0}^{t}\int_{\R^{d}\times\R}\int_{\R^{d}\times\R}f_{r}^{\epsilon,\delta}(x,\xi)f_{r}(y,\zeta)\bar{\rho}_{\delta}(\xi-\zeta)(\nabla\rho_{\epsilon}(x-y)\cdot(b(x,\xi)-b(y,\zeta))\nonumber \\
 & \qquad+\rho^{\epsilon}(x-y)\div_{y}\,b(y,\zeta))\varphi_{r}(x)\,dyd\zeta dxd\xi dr.\nonumber 
\end{align}

\textit{Step 3}: The last addend in the right hand side above is the
commutator error, which by Lemma \ref{lem:commutator} goes to zero
in $L_{t,\omega}^{2}$ letting first $\delta\rightarrow0$ and then
$\epsilon\rightarrow0$. Therefore, taking the $L_{t,\omega}^{2}$-limit
in \eqref{eq:ineq_eps} first for $\delta\rightarrow0$ then for $\epsilon\rightarrow0$,
we obtain the statement.
\end{proof}
We are ready to prove the key Lemma \ref{lem:key_lemma}.
\begin{proof}
[Proof of Lemma \ref{lem:key_lemma}]By Lemma \ref{lem:eq_f2} we
have, for every nonnegative test function $\varphi$ in $C_{c}^{\infty}([0,T]\times\R^{d})$
independent of $\xi$, for a.e. $t$ (with the exceptional set possibly
depending on $\varphi$),
\begin{align}
\E\int_{\R^{d}\times\R}(|f_{t}|-f_{t}^{2})\varphi_{t}\,dxd\xi & \le\int_{\R^{d}\times\R}(|f_{0}|-f_{0}^{2})\varphi_{0}\,dxd\xi\nonumber \\
 & +\E\int_{0}^{t}\int_{\R^{d}\times\R}[\partial_{t}\varphi+\frac{1}{2}\Delta\varphi+\div(b(x,\xi)\varphi)](|f|-f^{2})\,dxd\xi dr;\label{eq:ineq_gen_sol}
\end{align}
here we used that $\int_{0}^{t}\int_{\R^{d}\times\R}\nabla\varphi(|f|-f^{2})\,dxd\xi dW_{r}$
is an $L^{2}$ martingale with zero mean, since $\nabla\varphi(|f|-f^{2})$
is bounded and compactly supported.

The idea at this point is to use duality; that is, we would like to
take a test function $\varphi$, independent of $\xi$, nonnegative
and sufficiently regular, with $\varphi_{T}>0$, such that, for every
$\xi$ in a bounded interval $[-R,R]$,
\begin{align}
 & \partial_{t}\varphi+\frac{1}{2}\Delta\varphi+\div(b(x,\xi)\varphi)\le C.\label{PDE phi ineq-2}
\end{align}
Then we could conclude by Gronwall's inequality. To do so, the strategy
is as follows. First we take $\varphi$ as a nonnegative solution
to
\begin{align*}
 & \partial_{t}\varphi+\frac{1}{2}\Delta\varphi+F(x)\varphi=0\,,\hspace{1cm}\varphi(t_{fin},x)=1,
\end{align*}
with $F(x)=\|\div\,b(x,\cdot)\|_{L_{\xi,B_{R}}^{\infty}}$ (measurable
function), $t_{fin}$ a given time and $R$ such that the support
of $f$ is in $[0,T]\times\Omega\times\R^{d}\times[-R,R]$. Then we
use a bound on the transport term $b\cdot\nabla\varphi$ to obtain
\eqref{PDE phi ineq-2}.

For technical reasons, we take, for $\epsilon,t_{fin}>0$ fixed, $\varphi^{\epsilon}$
to be a solution on $[0,t_{fin}]$ to
\begin{align}
 & \partial_{t}\varphi^{\epsilon}+\frac{1}{2}\Delta\varphi^{\epsilon}+F^{\epsilon}\varphi^{\epsilon}=0\,,\hspace{1cm}\varphi^{\epsilon}(t_{fin},x)=\psi_{1/\epsilon}(x);\label{eq:PDE_dual}
\end{align}
here $\psi_{1/\epsilon}$ is a $C_{c}^{\infty}$ nonnegative function,
with values in $[0,1]$, equal to $1$ on $B_{1/\epsilon}(0)$ and
uniformly bounded (in $\epsilon$) in the $W^{1,\infty}(\R^{d})$
norm; $F^{\epsilon}$ is a compactly supported regularization of $F$,
converging to $F$ a.e. and in $L^{p}$, if $p<\infty$, or a.e. and
with uniform $L^{\infty}$ bound, if $p=\infty$. We extend $\varphi^{\epsilon}$
to the whole interval $[0,T]$ by taking $\varphi^{\epsilon}(t,x)=\psi_{1/\epsilon}(x)$
for $t\in[t_{fin},T]$. By Remark \ref{rmk:W2_reg} below, $\varphi^{\epsilon}$
is nonnegative and in $L_{t}^{\infty}(W_{x}^{2,\infty})\cap L_{x}^{\infty}(W_{t}^{1,\infty})$
for every $\epsilon>0$. Therefore, reasoning as in Remark \ref{rmk:enlarged_test},
$\varphi^{\epsilon}$ can be used as test function in \eqref{eq:ineq_gen_sol}.
Consequently, we have, for a.e. $t\le t_{fin}$, with the exceptional
set $N^{\epsilon,t_{fin}}$ possibly depending on $\epsilon$ and
$t_{fin}$,
\begin{align}
 & \E\int\varphi_{t}^{\epsilon}(|f_{t}|-f_{t}^{2})\,dxd\xi\nonumber \\
 & \le\int_{\R^{d}\times\R}\varphi_{0}^{\epsilon}(|f_{0}|-f_{0}^{2})\,dxd\xi\nonumber \\
 & \ +\int_{0}^{t}\E\int\big[\partial_{t}\varphi^{\epsilon}+\frac{1}{2}\Delta\varphi^{\epsilon}+F^{\epsilon}\varphi^{\epsilon}\big](|f|-f^{2})\,dxd\xi dr\nonumber \\
 & \ +\int_{0}^{t}\E\int[b\cdot\nabla\varphi^{\epsilon}+(\div\,b)\varphi^{\epsilon}-F\varphi^{\epsilon}](|f|-f^{2})\,dxd\xi dr\nonumber \\
 & \ +\int_{0}^{t}\E\int[F\varphi^{\epsilon}-F^{\epsilon}\varphi^{\epsilon}](|f|-f^{2})\,dxd\xi dr\label{eq:ineq_gen_sol_eps}\\
 & \le\int_{0}^{t}\E\int(b\cdot\nabla\varphi^{\epsilon})(|f|-f^{2})\,dxd\xi dr\nonumber \\
 & \ +\int_{0}^{t}\E\int(F-F^{\epsilon})\varphi^{\epsilon}(|f|-f^{2})\,dxd\xi dr,\nonumber 
\end{align}
where we have used that $|f|-f^{2}\ge0$ and that $f$ is supported
on $[0,T]\times\Omega\times\R^{d}\times[-R,R]$.

Before passing to the limit $\epsilon\rightarrow0$, we aim to replace
$t$ by $t_{fin}$ in the above inequality. This is not immediate,
since the function $t\mapsto\E[|f_{t}|-f_{t}^{2}]$ is not known to
be (even weakly) continuous. To overcome this difficulty, we fix a
version of the map $[0,T]\rightarrow L^{1}(\R^{d}\times\R)$ given
by $t\mapsto\E[|f_{t}|-f_{t}^{2}]$ and we use Lusin's theorem for
separable Banach space-valued functions, see for example Loeb, Talvila
\cite{LoeTal2004}: for every $\delta>0$, there exists a measurable
set $A_{\delta}\subseteq[0,T]$ with Lebesgue measure $|A_{\delta}|\ge T-\delta$,
such that $t\mapsto E[|f_{t}|-f_{t}^{2}]$ is continuous on $A_{\delta}$
as an $L^{1}(\R^{d}\times\R)$-valued map. We can also assume that
$A_{\delta}$ has no points which are isolated from the left (here
we say that $t_{0}$ is isolated from the left in $A_{\delta}$ if
$(t_{0}-\eta,t_{0})\cap A_{\delta}=\emptyset$ for some $\eta>0$):
indeed, the set of points of $A_{\delta}$ which are isolated from
the left is at most countable and thus has zero Lebesgue measure.
Therefore, for $t_{fin}\in A_{\delta}$, we can find a sequence $t_{n}\le t_{fin}$
in $A_{\delta}\setminus N^{\epsilon,t_{fin}}$ converging to $t_{fin}$
(as $n\rightarrow\infty$) and such that \eqref{eq:ineq_gen_sol_eps}
holds for $t_{n}$ and $\E[|f_{t_{n}}|-f_{t_{n}}^{2}]\rightarrow\E[|f_{t_{fin}}|-f_{t_{fin}}^{2}]$
in $L^{1}(\R^{d}\times\R)$. Moreover, by Remark \ref{rmk:W2_reg}
$\varphi^{\epsilon}$ is in $L_{x}^{\infty}(W_{t}^{1,\infty})$ and
so the map $[0,T]\rightarrow L^{\infty}(\R^{d}\times\R)$ given by
$t\mapsto\varphi_{t}^{\epsilon}$ is continuous. Hence, by Hölder's
inequality,
\begin{align*}
\E\int\varphi_{t_{n}}^{\epsilon}(|f_{t_{n}}|-f_{t_{n}}^{2})\,dxd\xi\rightarrow\E\int\varphi_{t_{fin}}^{\epsilon}(|f_{t_{fin}}|-f_{t_{fin}}^{2})\,dxd\xi.
\end{align*}
Since the right hand side of \eqref{eq:ineq_gen_sol_eps} is continuous
in time, we can pass to the limit in \eqref{eq:ineq_gen_sol_eps}
for $t_{n}\to t_{fin}$ and obtain \eqref{eq:ineq_gen_sol_eps} for
$t_{fin}\in A_{\delta}$. Since this is true for any $\delta>0$,
we obtain \eqref{eq:ineq_gen_sol_eps} for a.e. $t=t_{fin}$.

Now we let $\epsilon$ go to $0$. By Lemma \ref{lem:PDE_est}, applied
to the backward PDE \eqref{eq:PDE_dual}, and the uniform bound on
$F^{\epsilon}$ in $L^{p}$, we have a uniform (in $\epsilon$) bound
on $\|\varphi^{\epsilon}\|_{L_{t}^{\infty}(W_{x}^{1,\infty})}$. Therefore,
we can bound the first addend of the right hand side in \eqref{eq:ineq_gen_sol_eps}
by
\begin{align}
 & \limsup_{\epsilon\rightarrow0}\int_{0}^{t_{fin}}\E\int(b\cdot\nabla\varphi^{\epsilon})(|f|-f^{2})\,d\xi dx\nonumber \\
 & \le\|b\|_{L_{\xi,[-R,R]}^{\infty}(L_{x}^{\infty})}\sup_{\epsilon}\|\varphi^{\epsilon}\|_{L_{t}^{\infty}(W_{x}^{1,\infty})}\int_{0}^{t_{fin}}\E\int(|f|-f^{2})\,dxd\xi dr\nonumber \\
 & \le C\int_{0}^{t_{fin}}\E\int(|f|-f^{2})\,dxd\xi dr.\label{eq:b_bdd}
\end{align}
Concerning the second addend in \eqref{eq:ineq_gen_sol_eps}, in the
case $p<\infty$, $F-F^{\epsilon}$ converges to $0$ in $L_{x}^{p}$
and thus in $L^{p}([0,T]\times\Omega\times\R^{d}\times[-R,R])$; $\varphi^{\epsilon}$
is uniformly bounded in $L_{t,x}^{\infty}$ and $(|f|-f^{2})$ is
in $L^{p'}([0,T]\times\Omega\times\R^{d}\times[-R,R])$ by Remark
\ref{rmk:enlarged_test}. Therefore, by Hölder's inequality,
\begin{align*}
\limsup_{\epsilon\rightarrow0}\int_{0}^{t_{fin}}\E\int(F-F^{\epsilon})\varphi^{\epsilon}(|f|-f^{2})\,dxd\xi dr & =0.
\end{align*}
In the case $p=\infty$ we get the same result: here we exploit (via
dominated convergence theorem) the a.e.~convergence to $0$ and the
uniform bound of $\varphi^{\epsilon}(F-F^{\epsilon})$ and the fact
that $(|f|-f^{2})$ is in $L^{1}([0,T]\times\Omega\times\R^{d}\times[-R,R])$.
Finally, concerning the initial condition, using again the uniform
bound from Lemma \ref{lem:PDE_est} we get $\int_{\R^{d}\times\R}\varphi_{0}^{\epsilon}(|f_{0}|-f_{0}^{2})\,dxd\xi\le C\int_{\R^{d}\times\R}(|f_{0}|-f_{0}^{2})\,dxd\xi$.
Putting all together we have, for a.e. $t_{fin}>0$,
\begin{align*}
\E\int(|f_{t_{fin}}|-f_{t_{fin}}^{2})\,d\xi dx\le & C\int_{\R^{d}\times\R}(|f_{0}|-f_{0}^{2})\,dxd\xi\\
 & +C\int_{0}^{t_{fin}}\E\int(|f|-f^{2})\,d\xi dx\,dr.
\end{align*}
We conclude by Gronwall's lemma for discontinuous functions (cf.~Ethier,
Kurtz \cite[Theorem 5.1 in the Appendix]{EthKur1986}) that, for a.e.~$t\in[0,T]$,
\begin{align*}
\E\int(|f_{t}|-f_{t}^{2})\,d\xi dx & \le C\int_{\R^{d}\times\R}(|f_{0}|-f_{0}^{2})\,dxd\xi,
\end{align*}
where $C$ is a constant that depends only on the bound \eqref{eq:b_bdd}
and on the a priori estimates in Lemma \ref{lem:PDE_est}, applied
to the backward PDE \eqref{eq:PDE_dual}. Therefore, $C$ depends
only on $T$, $\|b\|_{L_{\xi,[-R,R]}^{\infty}(L_{x}^{\infty})}$ and
$\|\div\,b\|_{L_{x}^{p}(L_{\xi,[-R,R]}^{\infty})}$. The proof is
complete.
\end{proof}
Finally we prove Proposition \ref{prop:reconstruction}.
\begin{proof}
[Proof of Proposition \ref{prop:reconstruction}]Since $f_{0}$ takes
values in $\{0,\pm1\}$, we have $|f_{0}|-f_{0}^{2}=0$. Therefore,
Lemma \ref{lem:key_lemma} implies $f^{2}-|f|=0$ a.s.~(recall $|f|\le1$
by definition) and thus $f$ takes values in $\{0,\pm1\}$ for a.e.~$(t,\omega,x,\xi)$.
We then define $u(t,\omega,x):=\int_{\R}f(t,\omega,x,\xi)\,d\xi$.
Note that $u$ is well-defined since $f$ is compactly supported in
$\xi$ and measurable by Fubini's theorem.

Now we claim that, for every $h>0$, for a.e. $(t,\omega,x,\xi)$,
\begin{align}
 & (f(t,\omega,x,\xi)-f(t,\omega,x,\xi+h))(1_{-\infty<\xi<-h}+1_{h<\xi<+\infty})\ge0,\label{eq:f_nonincr}\\
 & f(t,\omega,x,\xi)-f(t,\omega,x,\xi+h)+1\ge0.\label{eq:f_jump}
\end{align}
Leaving the proof of these inequality for later, we use them to conclude.
Since the pushforward of the Lebesgue measure via the map $(\xi,h)\mapsto(\xi,\xi+h)$
is equivalent to the Lebesgue measure, the two inequalities above
imply, for a.e.\ $(t,\omega,x,\xi,\eta)$,
\begin{align}
 & (f(t,\omega,x,\xi)-f(t,\omega,x,\eta))(1_{\xi<\eta<0}+1_{0<\xi<\eta})\ge0,\label{eq:f_nonincr_1}\\
 & (f(t,\omega,x,\xi)-f(t,\omega,x,\eta)+1)1_{\xi<\eta}\ge0.\label{eq:f_jump_1}
\end{align}
Now we fix a version of $f$ and we consider, for fixed $(t,\omega,x)$,
the set $A=A(t,\omega,x)=\{\xi<0:\,(f(t,\omega,x,\xi)-f(t,\omega,x,\eta))\sgn(\xi-\eta)\le0\text{ for a.e. }\eta<0\}$.
By Fubini's theorem, \eqref{eq:f_nonincr_1} implies that, for a.e.\ $(t,\omega,x)$,
$A(t,\omega,x)$ is a full-measure set on $(-\infty,0)$. Moreover,
for any $(t,\omega,x)$, $f$ is non-increasing on $A(t,\omega,x)$.
Indeed, if this would not be true, we could find $\xi<\eta$ in $A$
with $f(t,\omega,x,\xi)-f(t,\omega,x,\eta)<0$. Thus, since $f(t,\omega,x,\xi)-f(t,\omega,x,\zeta)\ge0$
for a.e.~$\zeta>\xi$, we obtain $f(t,\omega,x,\zeta)-f(t,\omega,x,\eta)<0$
for a.e.~$\zeta\in(\xi,\eta)$, in contradiction to $\eta\in A$.
Similarly, for a.e.\ $(t,\omega,x)$, $B(t,\omega,x)=\{\xi>0:\,(f(t,\omega,x,\xi)-f(t,\omega,x,\eta))\sgn(\xi-\eta)\le0\text{ for a.e. }\eta>0\}$
is a full-measure set on $(0,+\infty)$ on which $f$ is non-increasing.
Since $f$ is compactly supported in $\xi$ and takes values a.e.\ in
$\{0,\pm1\}$, we conclude for a.e.\ $(t,\omega,x)$, $f=-1_{\{a<\xi<0\}}+1_{\{0<\xi<b\}}$
for some $a\le0\le b$ (depending on $(t,\omega,x)$) on the full-measure
set $A(t,\omega,x)\cup B(t,\omega,x)$. By \eqref{eq:f_jump_1} this
yields that either $f=-1_{\{a<\xi<0\}}$ a.e.\ or $f=-1_{\{0<\xi<b\}}$
a.e.~and thus $f=\chi(u)$ a.e.. Progressive measurability of $u$
follows from the respective property of $f=\chi(u)$, by Remark \ref{rmk:meas_u_chi}.

In remains to prove the claim above, that is, \eqref{eq:f_nonincr}
and \eqref{eq:f_jump}. To prove \eqref{eq:f_nonincr} we take a nonnegative
test function $\psi$ in $C_{c}^{\infty}([0,T]\times\R^{d}\times\R)$
with support contained in $(-\infty,-h)$. We call $\varphi$ the
function such that $\psi=-\partial_{\xi}\varphi$ and that $\varphi(-a)=0$
for $a$ large enough; $\varphi$ is a nonpositive nonincreasing function,
constant on $[-h,+\infty).$ We then have by \eqref{eq:gen_kinetic_measure},
for a.e.~$\omega$,
\begin{align*}
\int & (f(t,x,\xi)-f(t,x,\xi+h))\psi(t,x,\xi)\,dxd\xi dt\\
 & =\int f(t,x,\xi)(\psi(t,x,\xi)-\psi(t,x,\xi-h))\,dxd\xi dt\\
 & =\int(\varphi(t,x,0)-\varphi(t,x,-h))\,dxdt-\int(\varphi(t,x,\xi)-\varphi(t,x,\xi-h))\,\nu(dx,d\xi,dt)\\
 & =-\int(\varphi(t,x,\xi)-\varphi(t,x,\xi-h))\,\nu(dx,d\xi,dt)\ge0.
\end{align*}
This proves that $f(t,\omega,x,\xi)-f(t,\omega,x,\xi+h)\ge0$ on $\{-\infty<\xi<-h\}$;
similarly for $\{h<\xi<+\infty\}$. This proves \eqref{eq:f_nonincr}.

For \eqref{eq:f_jump}, we take a nonnegative test function $\psi$
in $C_{c}^{\infty}([0,T]\times\R^{d}\times\R)$ and we call $\varphi$
the nonpositive, nonincreasing function such that $\psi=-\partial_{\xi}\varphi$
and that $\varphi(-a)=0$ for $a$ large enough. Again we have by
\eqref{eq:gen_kinetic_measure}
\begin{align*}
\int & (f(t,x,\xi)-f(t,x,\xi+h)+1)\psi(t,x,\xi)\,dxd\xi dt\\
 & =\int f(t,x,\xi)(\psi(t,x,\xi)-\psi(t,x,\xi-h))\,dxd\xi dt+\int(\int\psi(t,x,\xi)\,d\xi)\,dxdt\\
 & =\int(\varphi(t,x,0)-\varphi(t,x,-h))\,dxdt-\int(\varphi(t,x,\xi)-\varphi(t,x,\xi-h))\,\nu(dx,d\xi,dt)\\
 & \quad+\int(\varphi(t,x,-R_{1})-\varphi(t,x,R_{1}))\,dxdt,
\end{align*}
for some $R_{1}$ such that the support of $\psi$ is contained in
$[0,T]\times\R^{d}\times[-R_{1},R_{1}]$. Now the monotonicity property
of $\varphi$ gives that $\varphi(t,x,\xi)-\varphi(t,x,\xi-h)\le0$
for every $\xi$, and also $\varphi(t,x,-R_{1})-\varphi(t,x,-h)\ge0$
and $\varphi(t,x,0)-\varphi(t,x,R_{1})\ge0$. Therefore, the right
hand side of the formula above is $\ge0$. This proves \eqref{eq:f_jump}
and concludes the proof of the claim.
\end{proof}

\section{Appendix A: a priori estimates on parabolic PDEs\label{sec:PDE}}

In this section we provide a priori estimates for a linear parabolic
PDE on $\R^{d}$ of the form
\begin{align}
\partial_{t}\varphi=\frac{1}{2}\Delta\varphi+g\varphi+h\varphi\,,\label{PDE eq}
\end{align}
where $g\in L_{x}^{p}$ for some finite $p>d$ and $h\in L_{x}^{\infty}$
and the initial datum $\varphi_{0}$ is nonnegative. Since we are
interested in a priori estimates in this section, we suppose that
$g$, $h$ and $\varphi_{0}$ are smooth and compactly supported.
The estimates can be applied also to the backward PDE, by a change
of time. The methods used in this section are essentially classical,
see for example (among many other references) Krylov \cite{Kry2008}
for heat kernel estimates in $L^{p}$ spaces and Fedrizzi, Flandoli
\cite{FedFla2011} for estimates on Kolmogorov-type PDEs.
\begin{rem}
\label{rmk:W2_reg}The existence of a nonnegative solution $\varphi$
in $L_{t}^{\infty}(W_{x}^{2,\infty})$ to \eqref{PDE eq} in the case
of smooth compactly supported coefficients and nonnegative initial
datum is ensured, for example, by the representation formula 
\begin{equation}
\varphi(t,x)=\E[\exp[\int_{0}^{t}(g(x+W_{r}-W_{t})+h(x+W_{r}-W_{t}))\,dr]\varphi_{0}(x-W_{t})],\label{eq:FK}
\end{equation}
where the expectation $\E$ and Brownian motion $W$ are defined on
some generic probability space, not related to the one used before.
The equation also implies, again for smooth compactly supported data,
that such a solution is in $L_{x}^{\infty}(W_{t}^{1,\infty})$.
\end{rem}

We start by recalling the regularizing properties of the heat kernel,
of easy (and classical) proof:
\begin{lem}
Let $p_{t}(x)=t^{-d/2}p_{1}(t^{-1/2}x)$ be the heat kernel on $\R^{d}$,
i.e.~$p_{1}(x)=(2\pi)^{-d/2}e^{-|x|^{2}/2}$. Then we have, for $m\in[1,\infty],$
\begin{align*}
\|p_{t}\|_{L_{x}^{m}}\le C_{m,d}t^{-(d-d/m)/2}\quad\text{and}\quad & \|\nabla p_{t}\|_{L_{x}^{m}}\le C_{m,d}t^{-(1+d-d/m)/2}.
\end{align*}
\end{lem}

\begin{proof}
We only prove the second inequality, the proof of the first one being
similar. The case $m=\infty$ is obvious, thus let $m\in[1,\infty).$
Note that $\nabla p_{t}(x)=t^{-(1+d)/2}\nabla p_{1}(t^{-1/2}x)$.
By the change of variable $y=t^{-1/2}x$, we get
\begin{align*}
\int_{\R^{d}}|\nabla p_{t}(x)|^{m}\,dx & =t^{-(1+d)m/2}t^{d/2}\int_{\R^{d}}|\nabla p_{1}(y)|^{m}\,dy=t^{-m(1+d-d/m)/2}\|\nabla p_{1}\|_{L_{x}^{m}}^{m},
\end{align*}
which is the desired estimates.
\end{proof}
We write the PDE \eqref{PDE eq} using the variational formulation:
\begin{align*}
\varphi_{t}= & p_{t}*\varphi_{0}+\int_{0}^{t}p_{t-s}*(g\varphi_{s})\,ds+\int_{0}^{t}p_{t-s}*(h\varphi_{s})\,ds.
\end{align*}

\begin{lem}
\label{lem:PDE_est}There exists a locally bounded function $c=c(T,\|g\|_{L_{x}^{p}},\|h\|_{L_{x}^{\infty}})$
such that, for every $\varphi_{0}$ in $C_{c}^{\infty}$, it holds
\begin{align*}
\|\varphi_{t}\|_{W_{x}^{1,\infty}} & \le\|\varphi_{0}\|_{W_{x}^{1,\infty}}c(T,\|g\|_{L_{x}^{p}},\|h\|_{L_{x}^{\infty}}).
\end{align*}
\end{lem}

\begin{proof}
Here $C$ denotes any positive constant, which can change from line
to line, possibly depending on $T$, $p$ and $d$. We start with
the $L_{x}^{\infty}$ estimate. Using Young's inequality for convolutions
we get
\begin{align*}
\|\varphi_{t}\|_{L^{\infty}} & \le\|p_{t}*\varphi_{0}\|_{L^{\infty}}+\int_{0}^{t}\|p_{t-s}*(g\varphi_{s})\|_{L^{\infty}}\,ds+\int_{0}^{t}\|p_{t-s}*(h\varphi_{s})||_{L^{\infty}}\,ds\\
 & \le C\|\varphi_{0}\|_{L^{\infty}}+C\int_{0}^{t}(t-s)^{-d/2p}\|g\varphi_{s}\|_{L^{p}}\,ds+C\int_{0}^{t}\|h\varphi_{s}\|_{L^{\infty}}\,ds\\
 & \le C\|\varphi_{0}\|_{L^{\infty}}+C\int_{0}^{t}(t-s)^{-d/2p}\|g\|_{L^{p}}\|\varphi_{s}\|_{L^{\infty}}\,ds+C\int_{0}^{t}\|h\|_{L^{\infty}}\|\varphi_{s}\|_{L^{\infty}}\,ds.
\end{align*}
Since $p>d$, $(t-s)^{-d/2p}$ is locally in $L^{2}$, Hölder's inequality
yields
\begin{align*}
\|\varphi_{t}\|_{L^{\infty}} & \le C\|\varphi_{0}\|_{L^{\infty}}+C\|g\|_{L^{p}}(\int_{0}^{t}\|\varphi_{s}\|_{L^{\infty}}^{2}\,ds)^{1/2}+C\|h\|_{L^{\infty}}(\int_{0}^{t}\|\varphi_{s}\|_{L^{\infty}}^{2}\,ds)^{1/2}
\end{align*}
and thus
\begin{align*}
\|\varphi_{t}\|_{L^{\infty}}^{2} & \le C\|\varphi_{0}\|_{L^{\infty}}^{2}+C\|g\|_{L^{p}}^{2}\int_{0}^{t}\|\varphi_{s}\|_{L^{\infty}}^{2}\,ds+C\|h\|_{L^{\infty}}^{2}\int_{0}^{t}\|\varphi_{s}\|_{L^{\infty}}^{2}\,ds.
\end{align*}
Gronwall's inequality implies
\begin{align}
\|\varphi_{t}\|_{L^{\infty}} & \le C\|\varphi_{0}\|_{L^{\infty}}\exp[C(\|g\|_{L^{p}}^{2}+\|h\|_{L^{\infty}}^{2})].\label{eq:Linfty_estimate_PDE}
\end{align}
We continue with the $L_{x}^{\infty}$ estimate for $\nabla\varphi_{t}$.
Using again Young's inequality we get
\begin{align*}
 & \|\nabla\varphi_{t}\|_{L^{\infty}}\\
 & \le\|p_{t}*\nabla\varphi_{0}\|_{L^{\infty}}+\int_{0}^{t}\|\nabla p_{t-s}*(g\varphi_{s})\|_{L^{\infty}}\,ds+\int_{0}^{t}\|\nabla p_{t-s}*(h\varphi_{s})||_{L^{\infty}}\,ds\\
 & \le C\|\nabla\varphi_{0}\|_{L^{\infty}}+C\int_{0}^{t}(t-s)^{-(1+d/p)/2}\|g\varphi_{s}\|_{L^{p}}\,ds+C\int_{0}^{t}(t-s)^{-1/2}\|h\varphi_{s}\|_{L^{\infty}}\,ds\\
 & \le C\|\nabla\varphi_{0}\|_{L^{\infty}}+C\int_{0}^{t}(t-s)^{-(1+d/p)/2}\|g\|_{L^{p}}\|\varphi_{s}\|_{L^{\infty}}\,ds\\
 & +C\int_{0}^{t}(t-s)^{-1/2}\|h\|_{L^{\infty}}\|\varphi_{s}\|_{L^{\infty}}\,ds.
\end{align*}
Since $p>d$, $(t-s)^{-(1+d/p)/2}$ is locally integrable and we obtain,
with \eqref{eq:Linfty_estimate_PDE},
\begin{align*}
\|\nabla\varphi_{t}\|_{L^{\infty}} & \le C\|\nabla\varphi_{0}\|_{L^{\infty}}+C(\|g\|_{L^{p}}+\|h\|_{L^{\infty}})\|\varphi_{0}\|_{L^{\infty}}\exp[C(\|g\|_{L^{p}}^{2}+\|h\|_{L^{\infty}}^{2})].
\end{align*}
The proof is complete.
\end{proof}

\section{Appendix B: measurability\label{sec:Appendix}}

In the following, let $(E,\mathcal{E},\mu)$ be a $\sigma$-finite
measure space. For a function $f:E\rightarrow\R$ recall the definition
(given in the introduction) of measurability. Given a Banach space
$V$ and a function $f:E\rightarrow V$, we recall the following three
definitions of measurability of $f$:
\begin{itemize}
\item we say that $f$ is strongly measurable if it is the pointwise (everywhere)
limit of a sequence of $V$-valued simple measurable functions (i.e.
of the form $\sum_{i=1}^{N}v_{i}1_{A_{i}}$ for $A_{i}$ in $\mathcal{E}$
and $v_{i}$ in $V$);
\item we say that $f$ is weakly measurable if, for every $\varphi$ in
$V^{*}$, $x\mapsto\langle f(x),\varphi\rangle_{V,V^{*}}$ is measurable;
\item if $V=U^{*}$ is the dual space of a Banach space $U$, we say that
$f$ is weakly-{*} measurable if, for every $\varphi$ in $U$, $x\mapsto\langle f(x),\varphi\rangle_{V,U}$
is measurable;
\item we say that $f$ is Borel measurable if, for every open set $A$ in
$V$ (endowed with the strong topology), $f^{-1}(A)$ is in $\mathcal{E}$.
\end{itemize}
The following result is morally Pettis measurability theorem. The
present version is a consequence of \cite[Chapter I Propositions 1.9 and 1.10]{VTC87}.
\begin{prop}
\label{prop:strong_weak_meas}Let $V$ be a separable Banach space.
Then the notions of strong measurability, weak measurability and Borel
measurability coincide. They also coincide with the weak-{*} measurability
if moreover $V$ is reflexive (in particular if $V=\R$).
\end{prop}

As mentioned in the introduction, in the definition of $L^{p}$ spaces
we only consider two cases: (1) $V=U^{*}$ is the dual space of a
separable Banach space, where $L^{0}(E;V)$ is the space of equivalent
classes of weakly-{*} measurable functions; (2) $V$ is a separable
Banach space, where $L^{0}(E;V)$ is the space of equivalent classes
of weakly (or strongly or Borel) measurable functions. In both cases,
for any function $f$ in $L^{0}(E;V)$, the function $x\mapsto\|f(x)\|_{V}$
is measurable: in the case (1) because $\|f(x)\|=\sup_{\varphi\in D}|\langle f(x),\varphi\rangle|$,
where $D$ is a countable sense set of $B_{1}^{U}$ (the unit centered
ball in $U$); in the case (2) as composition of the Borel map $f$
and the continuous map $\|\cdot\|_{V}$. Therefore, it makes sense
to define the spaces $L^{p}(E;V)$ for $1\le p\le\infty$.
\begin{prop}
\label{prop:Lp_two_var}Let $D$ be a domain of $\R^{n}$. For every
$1\le p\le\infty$, the space $L^{p}(E\times D,\mathcal{E\otimes\mathcal{B}}(D))$
is canonically embedded in $L^{p}(E;L^{p}(D))$ (whose functions are
weakly measurable for $1\le p<\infty$, weakly-{*} measurable for
$p=\infty$). This embedding is a surjective isomorphism.
\end{prop}

\begin{proof}
The embedding result is easy to show using Fubini's theorem, we prove
only the surjectivity. We start with the case $p<\infty$. To prove
this, let $F$ be an element (more precisely, a representative of
an element) in $L^{p}(E;L^{p}(D))$. By Proposition \ref{prop:strong_weak_meas},
$F$ is strongly measurable, i.e.~there exists a sequence $(F_{n})_{n}$
of simple functions in $L^{p}(E;L^{p}(D))$ which converges to $F$
in $L^{p}(D)$ for every $x$ and, without loss of generality, in
$L^{p}(E;L^{p}(D))$. We can write $F_{n}$ as
\begin{align*}
 & F_{n}(x)=\sum_{k=1}^{N(n)}F_{n,k}1_{A_{n,k}}(x)
\end{align*}
for some measurable sets $A_{n,k}$ and some elements $F_{n,k}$ in
$L^{p}(D)$. Now we define, for each $n$, the map $G_{n}:E_{x}\times E_{y}\rightarrow V$
by
\begin{align*}
 & G_{n}(x,y)=\sum_{k=1}^{N(n)}G_{n,k}(y)1_{A_{n,k}}(x),
\end{align*}
where $G_{n,k}$ is a representative of $F_{n,k}$. The function $G_{n}$
is measurable in $(x,y)$; since $\|G_{n}-G_{m}\|_{L^{p}(E\times D)}=\|F_{n}-F_{m}\|_{L^{p}(E;L^{p}(D))}$,
the sequence $(G_{n})_{n}$ is Cauchy in $L^{p}(E\times D)$, therefore
it converges to some $G$ in $L^{p}(E\times D)$. In particular $x\mapsto[y\mapsto G_{n}(x,y)]$
(where $[y\mapsto G_{n}(x,y)]$ is the equivalence class of $y\mapsto G_{n}(x,y)$)
converges to $x\mapsto[y\mapsto G(x,y)]$ in $L^{p}(E;L^{p}(D))$.
It follows that $x\mapsto[y\mapsto G(x,y)]$ coincides with $F$ $\mu$-a.e..
Hence $G$ is the desired representative in $L^{p}(E\times D)$ of
$F$. This concludes the proof in the case $p<\infty$.

The case $p=\infty$ can be reduced to the case $p<\infty$. Indeed,
call $(E_{n})_{n}$ an increasing sequence of sets with finite measure
and with $E_{n}\nearrow E$; then any function $f$ in $L^{\infty}(E;L^{\infty}(D))$,
restricted to $L^{\infty}(E_{n};L^{\infty}(B_{R}\cap D))$, is also
a weakly measurable function in $L^{2}(E_{n};L^{2}(B_{R}\cap D))$.
Hence, it has a representative in $L^{2}(E_{n}\times(B_{R}\cap D))$
and thus in $L_{loc}^{2}(E\times D)$, by arbitrariness of $R$ and
$n$, and this representative is essentially bounded.
\end{proof}
The following result is in Valadier \cite{Val1990}, Theorem 2 (see
also Theorem A.4).
\begin{thm}
\label{thm:dual_Lp}Assume that $\mu$ is finite and $\mathcal{E}$
is $\mu$-complete. Let $S$ be a metric $\sigma$-compact locally
compact space and, for any $R>0$, denote by $L_{R}^{\infty}(E;\mathcal{M}_{+}(S))$
the subset of $L^{\infty}(E;\mathcal{M}(S))$ of nonnegative measure-valued
functions $g$ with $\|g\|_{L^{\infty}(E;\mathcal{M}(S))}\le R$.
Then \textup{$L_{R}^{\infty}(E;\mathcal{M}_{+}(S))$} is (embedded
isomorphically in) a bounded sequentially weakly-{*} closed subset
of the dual space of $L^{1}(E;C_{0}(S))$. In particular, every sequence
in $L_{R}^{\infty}(E;\mathcal{M}_{+}(S))$ admits a subsequence converging
weakly-{*} to an element of $L_{R}^{\infty}(E;\mathcal{M}_{+}(S))$.
\end{thm}

We close with a remark on operations on measurable functions:
\begin{rem}
\label{rmk:measurability_omega} (i) Assume that $V$ is the dual
space of a separable space $U$. Let $f:E\rightarrow V$ be a weakly-{*}
measurable map and let $\varphi:E\rightarrow U$ be a (weakly or equivalently
strongly) measurable map. Then the map $x\mapsto\langle f(x),\varphi(x)\rangle_{V,U}$
is measurable. Indeed, if $\varphi_{k}$ are simple measurable functions
approximating everywhere $\varphi$, then $x\mapsto\langle f(x),\varphi_{k}(x)\rangle_{V,U}$
are measurable functions approximating everywhere $x\mapsto\langle f(x),\varphi(x)\rangle_{V,U}$.

In particular, take a weakly-{*} progressively measurable function
$f:[0,T]\times\Omega\rightarrow L_{x,\xi}^{\infty}$ and a $\mathcal{P}\otimes\mathcal{B}(\R^{d})\otimes\mathcal{B}(\R)$-measurable
integrable function $\varphi:[0,T]\times\Omega\times\R^{d}\times\R\rightarrow\R$,
so that $\varphi:[0,T]\times\Omega\rightarrow L_{x,\xi}^{1}$ is a
progressively measurable function. Then $(t,\omega)\mapsto\langle f(t,\omega),\varphi(t,\omega)\rangle_{x,\xi}$
is a progressively measurable function.

(ii) An analogous property holds for bounded kinetic measures $m$.
In this case one can consider a more general class of test functions.
Let $\varphi:[0,T]\times\Omega\times\R^{d}\times\R\rightarrow\R$
be a measurable function (not an equivalence class) such that: 1)
for every $(x,\xi)$, $(t,\omega)\mapsto\varphi(t,\omega,x,\xi)$
is progressively measurable; 2) for a.e. $\omega$, the zero set being
independent of $(t,x,\xi)$, and for every $(x,\xi)$, $t\mapsto\varphi(t,\omega,x,\xi)$
is càdlàg or càglàd; 3) for a.e. $\omega$, the zero set being independent
of $(t,x,\xi)$, and for every $t$, $(x,\xi)\mapsto\varphi(t,\omega,x,\xi)$
is continuous; 4) for a.e. $\omega$, the map $(t,x,\xi)\mapsto\varphi(t,\omega,x,\xi)$
is bounded. Then, for every fixed representative of $m^{\omega}$,
the map 
\[
(t,\omega)\mapsto\int_{[0,t]\times\R^{d}\times\R}\varphi^{\omega}(r,x,\xi)m^{\omega}(r,x,\xi)\,dxd\xi dr
\]
is progressively measurable and has a.e.~càdlàg paths. The same result,
replacing càdlàg with càglàd in the thesis, holds for $\int_{[0,t)\times\R^{d}\times\R}\varphi^{\omega}m^{\omega}\,dxd\xi dr$.

We prove this fact first for $t\mapsto\varphi(t,\omega,x,\xi)$ càdlàg.
We take a regular function $\psi_{n}$ on $\R^{d}\times\R$ with $0\le\psi_{n}\le1$,
$\psi_{n}=1$ on~$B_{n}$ and with support on $B_{2n}$ and we define
$\varphi_{n}=n\psi_{n}\varphi*_{t}1_{[0,1/n]}$. Note that $\varphi_{n}$
is a.s.~in $C_{0}([0,T]\times\R^{d}\times\R)$: indeed $\varphi_{n}^{\omega}$
is Lipschitz continuous in $t$ uniformly in $(x,\xi)$ (by boundedness
of $\varphi^{\omega}$) and continuous in $(x,\xi)$ at $t$ fixed.
Moreover, for every $t$, $\omega\mapsto\varphi_{n}^{\omega}|_{[0,t]}\in C_{0}([0,t]\times\R^{d}\times\R)$
is weakly, hence strongly $\mathcal{F}_{t+1/n}$-measurable: indeed,
for every finite signed measure $\mu$ on $[0,t]\times\R^{d}\times\R$,
$\omega\mapsto\int_{[0,t]\times\R^{d}\times\R}\varphi_{n}^{\omega}\,d\mu=n\int_{[0,t+1/n]}\int_{[(s-1/n)\vee0,s\wedge t]}\varphi^{\omega}(s,x,\xi)\,d\mu(r,x,\xi)\,ds$
is $\mathcal{F}_{t+1/n}$-measurable by Fubini's theorem. The first
part of this remark then gives that $\int_{[0,t]\times\R^{d}\times\R}\varphi_{n}^{\omega}m^{\omega}\,dxd\xi dr$
is $\mathcal{F}_{t+1/n}$-measurable.

Now, for a.e.~$\omega,$ $\varphi_{n}^{\omega}$ converges everywhere
to $\varphi^{\omega}$, by the càdlàg and continuity properties of
$\varphi$. Therefore, for every $t$ fixed, $\int_{[0,t]\times\R^{d}\times\R}\varphi^{\omega}m^{\omega}\,dxd\xi dr$
is the a.s.~limit of $\int_{[0,t]\times\R^{d}\times\R}\varphi_{n}^{\omega}m^{\omega}\,dxd\xi dr$.
Hence $\int_{[0,t]\times\R^{d}\times\R}\varphi^{\omega}m^{\omega}\,dxd\xi dr$
is $\mathcal{F}_{t+1/n}$-measurable for every $n$ (recall that $\mathcal{F}_{0}$
is complete) and thus $\mathcal{F}_{t}$-measurable. Moreover, for
any fixed representative of $m^{\omega}$, for a.e.~$\omega$, the
map $t\mapsto\int_{[0,t]\times\R^{d}\times\R}\varphi^{\omega}m^{\omega}\,dxd\xi dr$
is càdlàg. Therefore, $(t,\omega)\mapsto\int_{[0,t]\times\R^{d}\times\R}\varphi^{\omega}m^{\omega}\,dxd\xi dr$
has the desired properties.

In the case $t\mapsto\varphi(t,\omega,x,\xi)$ càglàd, the same proof
applies but taking $\varphi_{n}=\psi_{n}\varphi*_{t}1_{[-1/n,0]}$
(with $\varphi_{t}$ extended as $\varphi_{0}$ for $-1/n\le t<0$).
\end{rem}

\section{Funding}

Benjamin Gess acknowledges financial support by the the Max Planck
Society through the Max Planck Research Group \textquotedblleft Stochastic
partial differential equations\textquotedblright{} and by the DFG
through the CRC \textquotedblleft Taming uncertainty and profiting
from randomness and low regularity in analysis, stochastics and their
applications\textquotedblright . Mario Maurelli has received support
via the research grant from the European Research Council under the
European Union\textquoteright s Seventh Framework Program (FP7/2007-2013)/ERC
grant agreement nr. 258237. Part of this research was done while Mario
Maurelli was a research assistant at Weierstrass Institute for Applied
Analysis and Stochastic, Berlin and at Technische Universität Berlin,
Germany.

\bibliographystyle{abbrv}
\bibliography{my_bib}

\def\cprime{$'$}
\begin{thebibliography}{10}

\bibitem{A04-1}
L.~Ambrosio.
\newblock Transport equation and {C}auchy problem for {$BV$} vector fields.
\newblock {\em Invent. Math.}, 158(2):227--260, 2004.

\bibitem{AKR10}
B.~Andreianov, K.~H. Karlsen, and N.~H. Risebro.
\newblock On vanishing viscosity approximation of conservation laws with
  discontinuous flux.
\newblock {\em Netw. Heterog. Media}, 5(3):617--633, 2010.

\bibitem{AKR11}
B.~Andreianov, K.~H. Karlsen, and N.~H. Risebro.
\newblock A theory of {$L^1$}-dissipative solvers for scalar conservation laws
  with discontinuous flux.
\newblock {\em Arch. Ration. Mech. Anal.}, 201(1):27--86, 2011.

\bibitem{AD15}
B.~Andreianov and D.~Mitrovi{\'c}.
\newblock Entropy conditions for scalar conservation laws with discontinuous
  flux revisited.
\newblock {\em Ann. Inst. H. Poincar\'e Anal. Non Lin\'eaire},
  32(6):1307--1335, 2015.

\bibitem{AF09}
S.~Attanasio and F.~Flandoli.
\newblock Zero-noise solutions of linear transport equations without
  uniqueness: an example.
\newblock {\em C. R. Math. Acad. Sci. Paris}, 347(13-14):753--756, 2009.

\bibitem{AF11}
S.~Attanasio and F.~Flandoli.
\newblock Renormalized solutions for stochastic transport equations and the
  regularization by bilinear multiplication noise.
\newblock {\em Comm. Partial Differential Equations}, 36(8):1455--1474, 2011.

\bibitem{BFGM14}
L.~Beck, F.~Flandoli, M.~Gubinelli, and M.~Maurelli.
\newblock Stochastic odes and stochastic linear pdes with critical drift:
  regularity, duality and uniqueness.
\newblock {\em arXiv:1401.1530}, 2014.

\bibitem{BM16}
O.~Butkovsky and L.~Mytnik.
\newblock Regularization by noise and flows of solutions for a stochastic heat
  equation.
\newblock {\em arXiv:1610.02553}, 2016.

\bibitem{Cat2016}
R.~Catellier.
\newblock Rough linear transport equation with an irregular drift.
\newblock {\em Stoch. Partial Differ. Equ. Anal. Comput.}, 4(3):477--534, 2016.

\bibitem{CG16}
R.~Catellier and M.~Gubinelli.
\newblock Averaging along irregular curves and regularisation of {ODE}s.
\newblock {\em Stochastic Process. Appl.}, 126(8):2323--2366, 2016.

\bibitem{ChePer2003}
G.-Q. Chen and B.~Perthame.
\newblock Well-posedness for non-isotropic degenerate parabolic-hyperbolic
  equations.
\newblock {\em Ann. Inst. H. Poincar\'e Anal. Non Lin\'eaire}, 20(4):645--668,
  2003.

\bibitem{CDCDP15}
G.~Crasta, V.~De~Cicco, and G.~De~Philippis.
\newblock Kinetic formulation and uniqueness for scalar conservation laws with
  discontinuous flux.
\newblock {\em Comm. Partial Differential Equations}, 40(4):694--726, 2015.

\bibitem{CDCDPG16}
G.~Crasta, V.~De~Cicco, G.~De~Philippis, and F.~Ghiraldin.
\newblock Structure of solutions of multidimensional conservation laws with
  discontinuous flux and applications to uniqueness.
\newblock {\em Arch. Ration. Mech. Anal.}, 221(2):961--985, 2016.

\bibitem{Dal2006}
A.-L. Dalibard.
\newblock Kinetic formulation for heterogeneous scalar conservation laws.
\newblock {\em Ann. Inst. H. Poincar\'e Anal. Non Lin\'eaire}, 23(4):475--498,
  2006.

\bibitem{FD14}
F.~Delarue and F.~Flandoli.
\newblock The transition point in the zero noise limit for a 1{D} {P}eano
  example.
\newblock {\em Discrete Contin. Dyn. Syst.}, 34(10):4071--4083, 2014.

\bibitem{DFV14}
F.~Delarue, F.~Flandoli, and D.~Vincenzi.
\newblock Noise prevents collapse of {V}lasov-{P}oisson point charges.
\newblock {\em Comm. Pure Appl. Math.}, 67(10):1700--1736, 2014.

\bibitem{DPL89}
R.~J. DiPerna and P.-L. Lions.
\newblock Ordinary differential equations, transport theory and {S}obolev
  spaces.
\newblock {\em Invent. Math.}, 98(3):511--547, 1989.

\bibitem{DR16}
R.~Duboscq and A.~R\'eveillac.
\newblock Stochastic regularization effects of semi-martingales on random
  functions.
\newblock {\em J. Math. Pures Appl. (9)}, 106(6):1141--1173, 2016.

\bibitem{EthKur1986}
S.~N. Ethier and T.~G. Kurtz.
\newblock {\em Markov processes}.
\newblock Wiley Series in Probability and Mathematical Statistics: Probability
  and Mathematical Statistics. John Wiley \& Sons, Inc., New York, 1986.
\newblock Characterization and convergence.

\bibitem{FedFla2011}
E.~Fedrizzi and F.~Flandoli.
\newblock Pathwise uniqueness and continuous dependence of {SDE}s with
  non-regular drift.
\newblock {\em Stochastics}, 83(3):241--257, 2011.

\bibitem{FF13}
E.~Fedrizzi and F.~Flandoli.
\newblock Noise prevents singularities in linear transport equations.
\newblock {\em J. Funct. Anal.}, 264(6):1329--1354, 2013.

\bibitem{F11}
F.~Flandoli.
\newblock {\em Random perturbation of {PDE}s and fluid dynamic models}, volume
  2015 of {\em Lecture Notes in Mathematics}.
\newblock Springer, Heidelberg, 2011.
\newblock Lectures from the 40th Probability Summer School held in Saint-Flour,
  2010.

\bibitem{FGP10}
F.~Flandoli, M.~Gubinelli, and E.~Priola.
\newblock Well-posedness of the transport equation by stochastic perturbation.
\newblock {\em Invent. Math.}, 180(1):1--53, 2010.

\bibitem{FGP11}
F.~Flandoli, M.~Gubinelli, and E.~Priola.
\newblock Full well-posedness of point vortex dynamics corresponding to
  stochastic 2{D} {E}uler equations.
\newblock {\em Stochastic Process. Appl.}, 121(7):1445--1463, 2011.

\bibitem{FGP13}
F.~Flandoli, M.~Gubinelli, and E.~Priola.
\newblock Remarks on the stochastic transport equation with {H}\"older drift.
\newblock {\em Rend. Semin. Mat. Univ. Politec. Torino}, 70(1):53--73, 2012.

\bibitem{FR02}
F.~Flandoli and M.~Romito.
\newblock Probabilistic analysis of singularities for the 3{D}
  {N}avier-{S}tokes equations.
\newblock In {\em Proceedings of {EQUADIFF}, 10 ({P}rague, 2001)}, volume 127,
  pages 211--218, 2002.

\bibitem{FR08}
F.~Flandoli and M.~Romito.
\newblock Markov selections for the 3{D} stochastic {N}avier-{S}tokes
  equations.
\newblock {\em Probab. Theory Related Fields}, 140(3-4):407--458, 2008.

\bibitem{GG16}
P.~Gassiat and B.~Gess.
\newblock Regularization by noise for stochastic {H}amilton-{J}acobi equations.
\newblock {\em arXiv:1609.07074}, 2016.

\bibitem{GS17-2}
B.~Gess and S.~Smith.
\newblock Stochastic continuity equations with conservative noise.
\newblock {\em arXiv:1710.04906}, 2017.

\bibitem{GS14}
B.~Gess and P.~E. Souganidis.
\newblock Scalar conservation laws with multiple rough fluxes.
\newblock {\em Commun. Math. Sci.}, 13(6):1569--1597, 2015.

\bibitem{GS14-2}
B.~Gess and P.~E. Souganidis.
\newblock Long-time behavior and averaging lemmata for stochastic scalar
  conservation laws.
\newblock {\em to appear in: Comm. Pure Appl. Math.}, pages 1--23, 2016.

\bibitem{GP93}
I.~Gy{\"o}ngy and {\'E}.~Pardoux.
\newblock On the regularization effect of space-time white noise on
  quasi-linear parabolic partial differential equations.
\newblock {\em Probab. Theory Related Fields}, 97(1-2):211--229, 1993.

\bibitem{Kru1970}
S.~N. Kru{\v{z}}kov.
\newblock First order quasilinear equations with several independent variables.
\newblock {\em Mat. Sb. (N.S.)}, 81 (123):228--255, 1970.

\bibitem{Kry2008}
N.~V. Krylov.
\newblock {\em Lectures on elliptic and parabolic equations in {S}obolev
  spaces}, volume~96 of {\em Graduate Studies in Mathematics}.
\newblock American Mathematical Society, Providence, RI, 2008.

\bibitem{Kun1984}
H.~Kunita.
\newblock Stochastic differential equations and stochastic flows of
  diffeomorphisms.
\newblock In {\em \'{E}cole d'\'et\'e de probabilit\'es de {S}aint-{F}lour,
  {XII}---1982}, volume 1097 of {\em Lecture Notes in Math.}, pages 143--303.
  Springer, Berlin, 1984.

\bibitem{Lec2010}
M.~L{\'e}cureux-Mercier.
\newblock Improved stability estimates for general scalar conservation laws.
\newblock {\em J. Hyperbolic Differ. Equ.}, 8(4):727--757, 2011.

\bibitem{LPS13}
P.-L. Lions, B.~Perthame, and P.~E. Souganidis.
\newblock Scalar conservation laws with rough (stochastic) fluxes.
\newblock {\em Stoch. Partial Differ. Equ. Anal. Comput.}, 1(4):664--686, 2013.

\bibitem{LPS14}
P.-L. Lions, B.~Perthame, and P.~E. Souganidis.
\newblock Scalar conservation laws with rough (stochastic) fluxes: the
  spatially dependent case.
\newblock {\em Stoch. Partial Differ. Equ. Anal. Comput.}, 2(4):517--538, 2014.

\bibitem{LPT94}
P.-L. Lions, B.~Perthame, and E.~Tadmor.
\newblock A kinetic formulation of multidimensional scalar conservation laws
  and related equations.
\newblock {\em J. Amer. Math. Soc.}, 7(1):169--191, 1994.

\bibitem{LoeTal2004}
P.~A. Loeb and E.~Talvila.
\newblock Lusin's theorem and {B}ochner integration.
\newblock {\em Sci. Math. Jpn.}, 60(1):113--120, 2004.

\bibitem{M10}
M.~Mariani.
\newblock Large deviations principles for stochastic scalar conservation laws.
\newblock {\em Probab. Theory Related Fields}, 147(3-4):607--648, 2010.

\bibitem{MMNPZ2013}
O.~Menoukeu-Pamen, T.~Meyer-Brandis, T.~Nilssen, F.~Proske, and T.~Zhang.
\newblock A variational approach to the construction and {M}alliavin
  differentiability of strong solutions of {SDE}'s.
\newblock {\em Math. Ann.}, 357(2):761--799, 2013.

\bibitem{MohNilPro2015}
S.-E.~A. Mohammed, T.~K. Nilssen, and F.~N. Proske.
\newblock Sobolev differentiable stochastic flows for {SDE}s with singular
  coefficients: applications to the transport equation.
\newblock {\em Ann. Probab.}, 43(3):1535--1576, 2015.

\bibitem{Pro2004}
P.~E. Protter.
\newblock {\em Stochastic integration and differential equations}, volume~21 of
  {\em Applications of Mathematics (New York)}.
\newblock Springer-Verlag, Berlin, second edition, 2004.
\newblock Stochastic Modelling and Applied Probability.

\bibitem{RevYor1999}
D.~Revuz and M.~Yor.
\newblock {\em Continuous martingales and {B}rownian motion}, volume 293 of
  {\em Grundlehren der Mathematischen Wissenschaften [Fundamental Principles of
  Mathematical Sciences]}.
\newblock Springer-Verlag, Berlin, third edition, 1999.

\bibitem{VTC87}
N.~N. Vakhania, V.~I. Tarieladze, and S.~A. Chobanyan.
\newblock {\em Probability distributions on {B}anach spaces}, volume~14 of {\em
  Mathematics and its Applications (Soviet Series)}.
\newblock D. Reidel Publishing Co., Dordrecht, 1987.
\newblock Translated from the Russian and with a preface by Wojbor A.
  Woyczynski.

\bibitem{Val1990}
M.~Valadier.
\newblock Young measures.
\newblock In {\em Methods of nonconvex analysis ({V}arenna, 1989)}, volume 1446
  of {\em Lecture Notes in Math.}, pages 152--188. Springer, Berlin, 1990.

\end{thebibliography}

\end{document}